\documentclass[11pt,reqno]{amsart}

\usepackage[english]{babel}
\usepackage{amssymb}
\usepackage{amsfonts}
\usepackage[T1]{fontenc}

\usepackage{tikz}
\usetikzlibrary{arrows,matrix}
\usetikzlibrary{positioning}

\usepackage{enumerate}
\usepackage{setspace}
\usepackage{nicefrac}
\usepackage{booktabs}
\usepackage{wasysym}

 \newtheorem{theorem}{Theorem}[section]
 \newtheorem{corollary}[theorem]{Corollary}
 \newtheorem{lemma}[theorem]{Lemma}
 \newtheorem{proposition}[theorem]{Proposition}
 
  \newtheorem{question}{\sc Question}
 \theoremstyle{definition}
 \newtheorem{definition}[theorem]{Definition}
 \theoremstyle{remark}
 \newtheorem{remark}[theorem]{Remark}
 \theoremstyle{definition}
 \newtheorem{example}[theorem]{Example}

\newcommand{\rk}{\mathrm{R}} 
\newcommand{\cactusR}{\mathrm{cR}}

\newcommand{\XX}{\mathbb{X}}

\newcommand{\gradR}[1]{\mathrm{R}_{_{\nabla^{#1}}}}
\newcommand{\gradc}[1]{\mathrm{cR}_{_{\nabla^{#1}}}}
 
 \newcommand{\ud}{{\underline{d}}}
 \newcommand{\ue}{{\underline{e}}}
 \renewcommand{\HF}{\mathrm{HF}} 
 \newcommand{\cat}{\mathrm{cat}}
 \newcommand{\cR}{\mathrm{cR}}
 \newcommand{\rank}{\mathrm{rank}}
 \renewcommand{\flat}{\mathrm{flat}}
 \newcommand{\flatKos}{\mathrm{flatKos}}
 \newcommand{\textbigotimes}{{\textstyle \bigotimes}}
 \newcommand{\udelta}{\underline{\delta}}
 \newcommand{\frakm}{\mathfrak{m}}
 \DeclareMathOperator{\reg}{reg}
 
 \DeclareMathOperator{\Spec}{Spec}

 \newcommand{\id}{\mathrm{id}}

 \newcommand{\bbN}{\mathbb{N}}
 \newcommand{\bbP}{\mathbb{P}}
 \newcommand{\bbX}{\mathbb{X}}
\newcommand{\bbY}{\mathbb{Y}}

 \newcommand{\rmR}{\mathrm{R}}

 \renewcommand{\tilde}[1]{\widetilde{#1}}
\renewcommand{\hat}[1]{\widehat{#1}}

 \newcommand{\vvirg}{ , \dots , }
\newcommand{\ootimes}{ \otimes \cdots \otimes }
\newcommand{\ttimes}{ \times \cdots \times }
\newcommand{\ooplus}{ \oplus \cdots \oplus }

\newcommand{\bfx}{\mathbf{x}}
\newcommand{\bfy}{\mathbf{y}}

\newcommand{\calF}{\mathcal{F}}

\newcommand{\partinto}[2][]{
      \ifthenelse{\equal{#1}{}}{
      {{\scalebox{1.5}[.8]{$\; \vdash$} {#2} }}}{
      {{\underset{\hfill #1}{\scalebox{1.5}[.8]{$\; \vdash$}} {\; #2} }}
      }}
\newcommand{\frakS}{\mathfrak{S}}

\DeclareMathOperator{\Ann}{Ann}

\newcommand{\textsum}{{\textstyle \sum}}

\newcommand{\textbinom}[2]{{\textstyle \binom{#1}{#2}}}
\newcommand{\textfrac}[2]{{\textstyle \frac{#1}{#2}}}

\title{Partially symmetric variants of Comon's problem via simultaneous rank}
  
\author[F. Gesmundo]{Fulvio Gesmundo}
\address[F. Gesmundo]{QMATH, University of Copenhagen, Universitetsparken 5, 2100 Copenhagen O., Denmark}
\email{fulges@math.ku.dk}

\author[A. Oneto]{Alessandro Oneto}
\address[A. Oneto]{FMA-IAG, Otto-von-Guericke Universit\"at Magdeburg, Universit\"atsplatz 2, 39106 Magdeburg, Germany}
\email{alessandro.oneto@ovgu.de, aless.oneto@gmail.com}

\author[E. Ventura]{Emanuele Ventura}
\address[E. Ventura]{Department of Mathematics, Texas A\&M University, College Station, TX 77843-3368, USA}
\email{eventura@math.tamu.edu, emanueleventura.sw@gmail.com}

\newcommand{\mysurjects}{\begin{tikzpicture} 
\node at (0,0) (a) {};
\node at (1.5,0) (b) {};
\path [draw,->>] (a) -- (b);
\end{tikzpicture}}

\newcommand\bsfrac[2]{
\scalebox{-1}[1]{\nicefrac{\scalebox{-1}[1]{$#1$}}{\scalebox{-1}[1]{$#2$}}}
}


\usepackage[left=3cm, right=3cm]{geometry}
\begin{document}
\maketitle

\begin{abstract}
A symmetric tensor may be regarded as a partially symmetric tensor in several different ways. These produce different notions of rank for the symmetric tensor which are related by chains of inequalities. By exploiting algebraic tools such as apolarity theory, we show how the study of the simultaneous symmetric rank of partial derivatives of the homogeneous polynomial associated to the symmetric tensor can be used to prove equalities among different partially symmetric ranks. This approach aims to understand to what extent the symmetries of a tensor affect its rank. We apply this to the special cases of binary forms, ternary and quaternary cubics, monomials, and elementary symmetric polynomials. 
\end{abstract}

\section{Introduction}\label{section: intro}
The problem of representing tensors in convenient ways is connected to several areas of pure and applied mathematics. A line of research concerns {\it additive decompositions}: given a tensor of order $d$, say $t \in V^{\otimes d}$, a {\it tensor decomposition} of $t$ is a sum of {\it rank-one} tensors, i.e., elements of the form $v_1\ootimes v_d$, adding up to $t$. The smallest length of such a decomposition of $t$ is the {\it tensor rank} of $t$. Whenever a tensor satisfies certain {\it symmetries}, it is natural to study tensor decompositions reflecting such symmetries. Thus, several possible notions of rank arise, which are usually referred to as ({\it partially}) {\it symmetric} tensor ranks. The study of partially symmetric tensors has recently gained interest; see, for instance, \cite{BerBraComMou:GenTensorDecompMomMatApplications, BalBerChrGes:PartiallySymRkW}.

The space of homogeneous polynomials, or \textit{forms}, of degree $d$ on a vector space $V^*$ can be naturally identified with the space of symmetric tensors in $V^{\otimes d}$; denote this space by $S^d V$. Symmetric tensor decompositions are classically known as {\it Waring decompositions}; these are sums of powers of linear forms. The corresponding rank is the {\it Waring rank}. This class of decompositions has been studied since the XIX century, when Sylvester completed the classification of binary forms in terms of their Waring rank \cite{Sylvester1852}, that is the case where $\dim V = 2$. A great amount of work is devoted to this topic; among others, we mention Clebsch \cite{Cle61:UeberCurven}, Lasker \cite{Las04:ZurTheorie}, Palatini \cite{Pal03:CubicaQuinaria, Pal03:FormeTernarie}, Terracini \cite{Ter:CoppieFormeTernarie, Ter15:RappresFormeQuaternarie}. During the last decades of the last century, Sylvester's ideas were re-read in the modern algebraic language of Apolarity Theory, see e.g. {\cite{Re92:SumRealForms,DK93:PolarCovariants,ER93:Apolarity,IarrKan:PowerSumsBook}}. A major breakthrough in the development of the subject was accomplished in 1995 by Alexander and Hirschowitz \cite{AlHi95}, who resolved the long standing problem of determining Waring ranks of generic forms in any number of variables and any degree {(see also \cite{Ia95:InverseSystem})}. Throughout the years, the Waring problem attracted the attention of a broader community and classical and modern tools from algebraic geometry as well as from other fields have been employed for a variety of questions in this subject; see, for instance, \cite{Kl99:ReprPoly,BrCoMoTs10:SymmTensor, ComSei11:RankBinary,CarCatGer:SolutionMonomials,
BuBuTe13:WaringMonomials}.

{Disregarding its} symmetries, a symmetric tensor can be regarded as an element of the space of partially symmetric tensors for different choices of partial symmetries and one can ask what are the relations among the different (partially symmetric) ranks which arise in this way. This was the object of a famous question raised by Comon, who asked whether the tensor rank of a symmetric tensor equals its symmetric rank; see \cite[Problem~15]{Oed:Report}. This problem received a great deal of attention in the last few years. Affirmative answers were derived under certain assumptions \cite{ComGolLimMou:SymmetricTensorsSymmetricTensorRank, BalBer:TensorRankTangentDevelopSegre, ZhangHuangQi:ComonsConjecture, Fri:RmksSymmetricRank, Seig:RanksSymRanksCubSurf}. Recently, Shitov gave an example for $d = 3$ and $\dim V = 800$, where Comon's question has negative answer \cite{Shitov:CounterexampleComon}. 

In this article, we approach a partially symmetric version of Comon's question investigating relations among the partially symmetric ranks of a symmetric tensor. Our results will be obtained via the study of \textit{simultaneous Waring decompositions} of the set of $k$-th partial derivatives of homogeneous polynomials. 

\subsection{Formulation of the problem}\label{technical}
Let $V$ be a vector space of dimension $n+1$ over an algebraically closed field $\Bbbk$ of characteristic zero.

Let $d \geq 0$ be an integer and $\ud = (d_1 \vvirg d_m) \in \bbN^{m}$ be a sequence of integers with $\sum_{i=1}^m d_i = d$. In this case, $\ud$ is called a {\it composition} of $d$ and denoted $\ud \partinto{d}$; write $\ud \partinto[m]{d}$ when the length of the composition is relevant.

Let $\{x_0,\ldots,x_n\}$ be a basis of $V$. 
\begin{itemize}
	\item For any multi-index $\alpha \in \bbN^{n+1}$, write 
\[
|\alpha| := \alpha_0+\ldots+\alpha_n \qquad \text{and} \qquad \alpha! := \alpha_0!\cdots\alpha_n!.
\]
\item The symmetric algebra $S^\bullet V=\bigoplus_{j\geq 0}S^jV$ of $V$ is $\Bbbk[x_0,\ldots,x_n]$, that is the ring of polynomials on $V^*$ with its standard grading.
	\item For each $d \in \bbN$, $S^d V$ is identified with the space of symmetric tensors in $V^{\otimes d}$; a basis of $S^d V$ is given by
	\[
	\{\bfx^\alpha:=x_0^{\alpha_0}\cdots x_n^{\alpha_n} ~|~ \alpha\in\bbN^{n+1},~|\alpha| = d \}.	 
	\]
	\item Let $\Bbbk[x_{i,j} : i = 1,\ldots,m, j= 0,\ldots,n]$ denote the ring of polynomials in $m(n+1)$ variables with the multigrading defined by $\deg(x_{i,j}) = \underline{e}_i := (0,\ldots,\underset{i}{1},\ldots,0)$. 
	\item For $\ud \partinto d$, let $S^\ud V := S^{d_1}V \ootimes S^{d_m}V \subseteq V^{\otimes d}$ denote the space of partially symmetric tensors: this is isomorphic to the space of polynomials of multi-degree~$\ud$.
	\item Denote by $\frakS_d$ the symmetric group of permutations on a set of $d$ elements; $\frakS_d$ acts on $V^{\otimes d}$ by permuting the tensor factors.
\end{itemize}

In summary, we have the following diagram:

\medskip

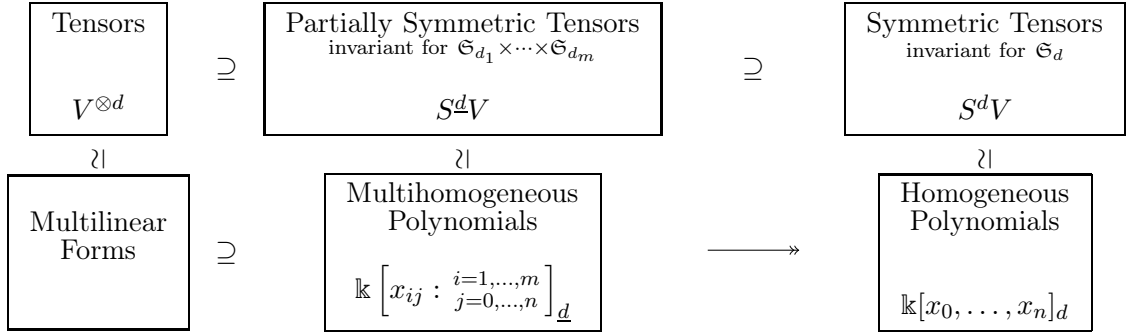
\resizebox{\textwidth}{!}{
\renewcommand{\arraystretch}{0.8}
$\begin{array}{ccccc} 
\boxed{ 
\begin{array}{c} \text{Tensors} \\ \\ \\ V^{\otimes d} \end{array}
} & \supseteq & 
\boxed{ 
\begin{array}{c} \text{Partially Symmetric Tensors} \\ \begin{smallmatrix} \text{invariant for} \ \frakS_{d_1} \ttimes \frakS_{d_m} \end{smallmatrix} \\ \\ S^\ud V \end{array}
} & \supseteq &
\boxed{ 
\begin{array}{c} \text{Symmetric Tensors} \\ \begin{smallmatrix} \text{invariant for} \ \frakS_{d} \end{smallmatrix} \\ \\ S^d V  \end{array}
} \\
\rule{0pt}{0.4cm}
\rotatebox[origin=c]{90}{$\simeq$} & & \rotatebox[origin=c]{90}{$\simeq$} & & \rotatebox[origin=c]{90}{$\simeq$} \\
\rule{0pt}{0.5cm} 
\boxed{ 
\begin{array}{c} \\ \text{Multilinear} \\ \text{Forms} \\ \\ \\ \end{array}
} & \supseteq & 
\boxed{ 
\begin{array}{c} \text{Multihomogeneous} \\ \text{Polynomials} \\ \\ \Bbbk \left[x_{ij}: \begin{smallmatrix} i = 1 \vvirg m \\ j = 0 \vvirg n\end{smallmatrix} \right]_\ud \end{array}
} & {\mysurjects} &
\boxed{ 
\begin{array}{c} \text{Homogeneous} \\ \text{Polynomials} \\ \\ \\ \Bbbk[x_0,\ldots,x_n]_d  \end{array}
} \\
\end{array}$
}

\medskip

The bottom right map is the surjective homomorphism defined by $x_{ij} \mapsto x_j$.

We give the following definition to formally introduce tensor decompositions respecting certain symmetries of a tensor.

\begin{definition}\label{defin: partially symmetric decomp}
Let $t \in S^\ud V$.  A {\bf partially symmetric tensor decomposition} of $t$ is a sum of rank-one partially symmetric tensors such that
\begin{equation}\label{eq: partsym decomp}
t = \sum_{i=1}^r v_{i,1}^{d_1}\ootimes v_{i,m}^{d_m}, \quad \text{ where } v_{i,j} \in V.
\end{equation}
The smallest $r$ such that a decomposition \eqref{eq: partsym decomp} exists is the {\bf partially symmetric rank} of $t$. Equivalently, given a multi-homogeneous polynomial $f$ of multi-degree $\ud$, a {\bf multi-homogeneous decomposition}, or {\bf $\ud$-decomposition}, of $f$ is a sum 
\begin{equation}\label{eq: multideg decomposition}
f = \sum_{i=1}^r \ell_{i,1}^{d_1}\cdots\ell_{i,m}^{d_m},\quad \text{ where } \deg(\ell_i) = {\underline{e}_i}.
\end{equation}
The smallest $r$ such that a decomposition \eqref{eq: multideg decomposition} exists is the {\bf $\underline{d}$-rank} of $f$, denoted $\rmR_\ud (f)$. 
\end{definition}

We will not distinguish between a multi-homogeneous polynomial and the corresponding partially symmetric tensor. In particular, we always write $f \in S^\ud V$ for a multihomogeneous polynomial $f$ of multi-degree $\ud$.

The space of symmetric tensors $S^dV$ is a subspace of $S^\ud(V)$ for any $\ud \partinto{d}$. Therefore, for any $\ud$ and for any $f \in S^dV$, we may ask the following.

\begin{question}\label{question}
Let $f \in S^dV$ and let $\ud \partinto d$. Is it true that $\rk_d(f) = \rk_\ud(f)$?
\end{question}

The original Comon's question corresponds to the case $\ud = (1,\ldots,1)$. Note that in the cases where the original question has an affirmative answer, so does the partially symmetric version for any $\ud$; see Lemma \ref{lemma: d geq d'}. Equivalently, an example where Question \ref{question} has a negative answer for some $\ud$ provides an example where the original Comon's question has negative answer as well. In this paper, we show instances where Question \ref{question} has affirmative answer for some choice of $\ud$, whereas the answer in the classical setting is not known. These are the cases of monomials and elementary symmetric polynomials.

Our approach is based on the study of \textit{simultaneous Waring decompositions} of a collection of homogeneous polynomials. The problem of determining simultaneous ranks dates back to Terracini, see \cite{Ter:CoppieFormeTernarie}; some related problems were addressed more recently in \cite{Fon:WaringProblemManyForms,AngGalMelOtt:WaringDecompPolyVector,CarVen:SimultaneousWaringRank}. The {\it simultaneous Waring rank} of a collection of homogeneous polynomials is the minimum number of linear forms needed to simultaneously write a Waring decomposition for every polynomial in the collection. In this work, given a polynomial $f$, we consider the simultaneous rank of the collection of its partial derivatives of a given order.

\begin{definition}\label{defin: gradient rank}
Let $f \in S^dV$ and let $k < d$. Let $\nabla^kf$ be the set of partial derivatives of order $k$ of $f$, i.e., $\nabla^kf = \left\{\frac{\partial^kf}{\partial \bfx^\alpha} : |\alpha| = k\right\}$. The {\bf $k$-th~gradient rank} of $f$ is the simultaneous rank of $\nabla^kf$, i.e.,
\[
	\gradR{k}(f) = \min\left\{r : \exists \ \ell_1,\ldots,\ell_r \in V \text{ such that } \frac{\partial^k f}{\partial \bfx^\alpha} = \sum_{i=1}^r c_{\alpha,i}\ell_i^{d-k}, \text{ for some } c_{\alpha,i} \in \Bbbk\right\}. 
\]
If $k = 1$, write $\gradR{}(f)$ for $\gradR{1}(f)$.
\end{definition}
Given $f \in S^dV$, for any $\ud \partinto[m]{d}$ with $d_m= d - k$, we have the following chain of inequalities, which is proven in Section \ref{subsec: ranks alg vars} (see Lemma \ref{lemma: d geq d'} and Corollary  \ref{corol: gradient rank as segre-veronese}): 
\begin{equation}\label{eq: intro}
	\rk_d(f) \geq \rk_\ud(f) \geq \gradR{k}(f). 
\end{equation}

In view of these inequalities, we will focus on the $k$-th gradient with the following strategy: we show that for certain families of homogeneous polynomials, the $k$-th gradient rank coincides with the Waring rank, so that \eqref{eq: intro} is a chain of equalities.

We prove our results employing classical {\it apolarity theory} which dates back to Sylvester, see Section \ref{subsec: apolarity}. Briefly, apolarity relates the rank of a symmetric tensor $f$ (respectively the simultaneous rank of a family of polynomials $f_1,\ldots,f_s$) to the minimal cardinality of a set of points whose ideal is contained in the {\it apolar ideal} of $f$ (respectively the intersection of the apolar ideals of the $f_i$'s); see Lemma \ref{lemma: apolarity homog} and Lemma \ref{remark: simul apolar}.

The notions of rank that we introduced have a {\it cactus} analog which will be denoted by $\cactusR$ with the corresponding subscripts; precise definitions will be given in Section \ref{subsec: cactus}. In terms of apolarity, this corresponds to studying \textit{any} $0$-dimensional scheme of minimal degree rather than just reduced sets of points.  This terminology was introduced in \cite{BuczBucz:SecantVarsHighDegVeroneseReembeddingsCataMatAndGorSchemes,RanSch11:RankForm}, but it coincides with the notion of \textit{scheme length} defined in \cite[Definition~4D]{Ia95:InverseSystem}). For cactus ranks there is a chain of inequalities analogous to \eqref{eq: intro} and we will use the same strategy explained above to study Question \ref{question} in this setting as~well.

Similarly, one can consider the notion of border rank, which is the \emph{upper semicontinuous closure} of the notion of rank: more precisely, the tensor border rank of $t \in V^{\otimes d}$ is the the minimum $r$ such that $t$ can be expressed as the limit of a sequence of $r$ tensors; partially symmetric border rank is defined similarly. In \cite{BGL:DeterminantalEquationsEKSConj, BL:RanksTensorsAndGeneralization}, some instances of Question \ref{question} for border rank are considered.

\subsection{State of the art: old and new}
We list some known results of the various versions of Comon's question described in the previous section and we present our main contributions. 

Let $f \in S^dV$, let $\ud \partinto[m]{d}$ and let $k = d-d_m$.
	\begin{itemize}
\item \underline{Binary forms.} If $\dim V = 2$, i.e., $f$ is a binary form, then Question \ref{question} has affirmative answer. The case $\ud = (1,\ldots,1)$ was proved in \cite[Corollary~3.12]{ZhangHuangQi:ComonsConjecture} and this implies an affirmative answer for any $\ud$ (see Lemma \ref{lemma: d geq d'}). Question \ref{question} has affirmative answer also if considered for \textit{cactus ranks}: indeed, in the case of binary forms the cactus rank coincides with the border rank and, in this case, Question \ref{question} admits a positive answer for border ranks (see \cite[Example~4.2.5]{BGL:DeterminantalEquationsEKSConj}). In Proposition \ref{prop: binary forms}, we prove
  \begin{align*}
    \gradR{k}(f) & = \min \{ \rmR_d(f) , d-k+1 \}, \\
    \quad \gradc{k}(f) & = \min \{ \cactusR_d(f) , d-k+1 \}. 
  \end{align*}
   Consequently, if $\rk_d(f) \leq d-k+1$ (respectively $\cactusR_d(f) \leq d-k+1$), then we have $\rk_d(f) = \rk_\ud(f) = \gradR{k}(f)$ (respectively $\cactusR_d(f) = \cactusR_\ud(f) = \gradc{k}(f)$).
  
\item \underline{Ternary and quaternary cubics}. Let $\dim V = 3$ or $4$ and $d = 3$, i.e., $f$ defines a plane cubic curve or a cubic surface. Then, \cite[Theorem~7.1(4)]{Fri:RmksSymmetricRank} (for ternary cubics) and \cite[Theorem~1.3]{Seig:RanksSymRanksCubSurf}  (for quaternary cubics) prove that that Question \ref{question} has affirmative answer for $\ud = (1,\ldots,1)$ and by Lemma \ref{lemma: d geq d'} this implies an affirmative answer for any $\ud$. \cite[Lemma 3.1]{Seig:RanksSymRanksCubSurf} (for ternary cubics) and \cite[Theorem 1.5]{Seig:RanksSymRanksCubSurf} (for quaternary cubics) prove that Question \ref{question} has an affirmative answer when interpreted for border ranks: in the range of interest, border rank coincides with cactus rank (see \cite[Sections 3.5 and 3.6]{BuBu15:OnDifference}), therefore Question \ref{question} has affirmative answer for cactus rank as well. In Corollary  \ref{corol: gradient rank as segre-veronese}, we additionally prove
\[
	\rk_{1,2}(f) = \gradR{}(f).
\]
In Proposition \ref{prop: gradient plane cubics} (for ternary cubics) and in Proposition \ref{prop: gradient cubic surfaces} (for quaternary cubics), we use a different method to prove the results described above and we prove 
\[
 	\cactusR_{1,2}(f) = \gradc{}(f). 
\]
Our proof provides additional information on the relation between minimal Waring and multi-homogeneous decompositions; see Remark \ref{rmk: simul decomp binary} Example \ref{ex: maximal rank cubic}.

\item \underline{Monomials}. Let $f = x_0^{\alpha_0}\cdots x_n^{\alpha_n}$ {be a monomial}.  We have the following:
		\begin{enumerate}[\rm (a)]
			\item {\rm (Theorem \ref{thm: gradient monomials})} if $k \leq \min_i\{\alpha_i\}$, then $\rk_d(f) = \rk_\ud(f) = \gradR{k}(f)$;
			\item {\rm (Theorem \ref{thm: cactus gradient monomials})} if $k = 1$, then $\cactusR_d(f) = \cactusR_\ud(f) = \cR_{_\nabla}(f)$.
			\end{enumerate}
\item \underline{Elementary symmetric forms of odd degree.} Let $d$ be odd. In Theorem \ref{thm: gradient elementary symmetric}, we prove that, if $f = \sum_{i_1< \cdots <i_d} x_{i_1}\cdots x_{i_d}$ {is the elementary symmetric polynomial of degree $d$ in $n+1$ variables}, then $\rmR_d(f) = \rmR_{1,d-1}(f) = \gradR{}(f)$.
	\end{itemize}

In addition, we provide a number of insights on different notions of rank and on their interplay. In particular, our approach suggests that classical apolarity theory is a valuable tool to study not only the symmetric rank of symmetric tensors, but also their partially symmetric ranks; indeed, the multigraded versions of apolarity theory present in the literature (see, e.g., \cite{Gal16:MultiApolarity,GalRanVil:ApolarityToric}), which are a priori more effective than the approach we are suggesting here, are generally more difficult to apply as they involve multigraded algebraic structures. Moreover, our approach proposes a systematic way to determine to what extent the symmetries of a tensor affect its rank. Moreover, by \eqref{eq: intro}, an example providing a negative answer to Question \ref{question} for certain $\underline{d} \partinto{d}$ would be an example where the original Comon's question has a negative answer. Hence, we expect that studying the intermediate steps of the entire hierarchy of partially symmetric ranks would provide new examples where the original Comon's question has a negative answer, besides the one presented in \cite{Shitov:CounterexampleComon}.

Besides a purely theoretical motivation, a better understanding of the problem posed in Question \ref{question} can provide insights from a computational point of view. Tensor rank decomposition has a number of applications in pure and applied mathematics, and determining optimal decompositions is a concrete problem for many applications. Throughout the last decades several algorithms have been developed to find optimal and nearly optimal tensor decompositions for given tensors, both in the exact and in the approximate setting. We refer to \cite[Chapters 7 and 9]{Hackbusch:TensorBook} for a discussion on these algorithms and to \cite{AllRho:PhylogeneticInvariantsGeneralMarkovModelSequenceMutation,Comon:IndependentCompAnalysis,GarStiStu:AlgebraicGeomBayNetworks,Lin:TensorProductsANOVAModels} for some examples of applications. In principle, given a symmetric tensor for which Comon's question has affirmative answer, then an optimal symmetric decomposition has the same length as an optimal (not necessarily symmetric) tensor decomposition: in applications where one is not interested in the symmetries of the decomposition but only in its length, one can use interchangeably algorithms for either the symmetric or the non-symmetric setting to compute an optimal decomposition. Similar reasoning can be applied in the case of partially symmetric decompositions.

\subsection{Structure of the paper} In Section \ref{section: basics}, we explain in more details the different notions of rank we are going to consider and we establish basic relations between them. In particular, we describe them in the framework of algebraic geometry. Moreover, we introduce algebraic tools from apolarity theory that we use in our computations. In Section \ref{sec: computations}, we prove our main results.

\section{Different notions of rank and apolarity}\label{section: basics}
In this section, we introduce basic definitions of the various notions of rank that we consider and we prove some relations among them. We also give the basics of apolarity theory which will be a fundamental tool for our approach.

Recall that $S^d V$ is the subspace of $V^{\otimes d}$ of symmetric tensors, namely tensors which are invariant under the action of the symmetric group $\frakS_d$ that permutes the tensor factors. Similarly, given a composition $\ud =(d_1,\ldots,d_m)\partinto{d}$, the space $S^\ud V := S^{d_1} V \ootimes S^{d_m}V$ is the subspace of partially symmetric tensors, namely tensors which are invariant under the action of the subgroup $\frakS_{d_1} \ttimes \frakS_{d_m} \subseteq \frakS_{d}$, where $V^{\otimes d} = \bigotimes_{j=1}^m V^{\otimes d_j}$ and $\frakS_{d_j}$ acts by permuting the tensor factors of $V^{\otimes d_j}$. In particular, a symmetric tensor may be regarded as a partially symmetric tensor, disregarding some of the additional symmetries. Hence, for any $\ud \partinto d$, we have the inclusions
\[
S^d V \subseteq S^\ud V \subseteq V^{\otimes d}.
\]
From a representation-theoretic point of view, $S^d V$ is the Cartan component of $S^\ud V$, under the diagonal action of ${GL}(V)$. 

Explicitly, for $f \in S^d V$, the \emph{polarization} of $f$ as a partially symmetric tensor in $S^\ud(V)$ (see e.g. \cite[Section~2.6.4]{Lan:TensorBook}) is given by the expression
\begin{equation}\label{eqn: polarization}
f = \frac{1}{d!}\sum_{\substack{ \alpha_1 \vvirg \alpha_{m-1} \in \bbN^{n+1} \\ \vert \alpha _i \vert = d_i }} \textbinom{d_1}{\alpha_1} \cdots \textbinom{d_{m-1}}{\alpha_{m-1}} \bfx^{\alpha_1} \ootimes \bfx^{\alpha_{m-1}} \otimes \frac{\partial^{d_1 + \cdots + d_{m-1}}} {\partial \bfx^{\alpha_1} \cdots \partial \bfx^{\alpha_{m-1}}} f .
\end{equation}
In the case $\ud = (1,d-1)$, which will be particularly interesting to us, it reduces to
\begin{equation}\label{eqn: polarization 1 d-1}
f = \frac{1}{d}\sum_{i=0}^n x_i \otimes \frac{\partial}{\partial x_i}f,
\end{equation}
that can be interpreted as a tensorial version of Euler's formula. 

\subsection{Ranks and projective varieties}\label{subsec: ranks alg vars}
Now, we include the notions of rank introduced in Definition \ref{defin: partially symmetric decomp} into the geometric framework of so-called \emph{$X$-ranks}. For any subset $\calF \subseteq \bbP^N$, let $\langle \calF \rangle$ denote the linear span of $\calF$, i.e., the smallest linear space containing~$\calF$.

\begin{definition}\label{def: X-rank}
 Let $X \subseteq \bbP^N$ be a non-degenerate projective variety and let $p \in \bbP^N$. The {\bf $X$-rank} of $p$, denoted $\rmR_X(p)$, is the minimal number of points of $X$ whose linear span contains the point $p$, i.e., the minimal $r$ such that $p \in \langle q_1 \vvirg q_r\rangle$ for some $q_1 \vvirg q_r \in X$.
\end{definition}
The notions of symmetric, tensor and partially symmetric rank introduced in Section \ref{section: intro} can be seen as $X$-ranks with respect to classical projective varieties such as Veronese, Segre and Segre-Veronese varieties, respectively.

For $\ud \partinto[m]{d}$, the \textit{$\ud$-th Segre-Veronese embedding} is the map 
\begin{align*}
 \nu_\ud : \bbP V \ttimes \bbP V &\to \bbP (S^{d_1} V \ootimes S^{d_m} V), \\
  ([v_1] \vvirg [v_m]) &\mapsto [v_1^{d_1} \ootimes v_m^{d_m}],
\end{align*}
where $[v]$ denotes the class of a vector $v\in V$ in the corresponding projective space. The image of the $\ud$-Segre-Veronese embedding is called {\it $\ud$-th Segre-Veronese variety}. When $\ud = (d)$, $\nu_d$ is the {\it Veronese embedding} of $\bbP V$ and its image is the {\it $d$-th Veronese variety}; when $\ud = (1,\ldots,1)$, $\nu_{(1,\ldots,1)}$ is the \textit{Segre embedding} of $\bbP V^{\times d}$ and its image is the \textit{Segre variety}. In particular, the $\ud$-rank of an element $f \in S^\ud V$ is the rank of $[f]$ with respect to the $\ud$-Segre-Veronese variety.

Fix two compositions $\ud,\ud' \partinto{d}$. Write $\ud \succeq \ud'$ if $\ud'$ is a refinement of $\ud$, in the sense that $\ud$ can be obtained from $\ud'$ by adding together some adjacent entries; more precisely, if $\ud = (d_1 \vvirg d_m)$ and $ \ud' = (d_1' \vvirg d_{m'}')$, then $\ud \succeq \ud'$ if and only if there exist $0 = s_0 < s_1 < \cdots < s_m = m'$ such that $ d_j = \sum _{i = s_{j-1} + 1} ^{s_j} d'_i$ for every $j$.

If $\ud \succeq \ud'$, then $S^\ud V \subseteq S^{\ud'}V$. Moreover, directly from the definition of Segre-Veronese varieties, we have the following.

\begin{lemma}\label{lemma: d geq d'}
Let $f \in S^d V$ and $\ud,\ud' \partinto{d}$ such that $\ud \succeq \ud'$. Then $\rmR_{\ud} (f) \geq \rmR_{\ud'}(f)$.
\end{lemma}
\begin{proof}
If $\ud \succeq \ud'$, then $\nu_{\ud}(\bbP V^{\times m}) \subseteq \nu_{\ud'}(\bbP V^{\times m'})$.
In particular, every set of points contained in $\nu_{\ud}(\bbP V^{\times m})$ whose linear span contains $[f]$ is also a set of points contained in $\nu_{\ud'}(\bbP V^{\times m'})$ whose linear span contains $[f]$; therefore, we obtain the desired inequality between the ranks.
\end{proof}

In fact, it is clear from its proof that Lemma \ref{lemma: d geq d'} holds for every element of $S^\ud V$. Here, we only deal with elements of $S^d V$, namely totally symmetric tensors. For this reason, the value $\rmR_\ud(f)$ does not depend on the order of the entries of $\ud$. Hence, one can consider an ordering similar to $\succeq$ on the set of partitions of $d$ and correspondingly one has an analog of Lemma \ref{lemma: d geq d'}. However, for the ease of notation, we keep working with compositions of integers rather than partitions.

The notion of simultaneous rank used in Definition \ref{defin: gradient rank} to define the gradient rank of $f \in S^d V$ can be generalized to the setting of $X$-rank as well.

\begin{definition}
Let $X \subseteq \bbP^N$ be a non-degenerate projective variety and let $\calF \subseteq \bbP^N$ be a subset. The {\bf simultaneous $X$-rank} of $\calF$, denoted $\rmR_X(\calF)$, is the minimal number of points on $X$ whose linear span contains $\calF$, i.e., the minimal $r$ such that there exist $q_1 \vvirg q_r \in X$ with $\calF \subseteq \langle q_1 \vvirg q_r \rangle$, or equivalently $\langle \calF \rangle \subseteq \langle q_1 \vvirg q_r \rangle$.
\end{definition}

In this general setting, we provide several elementary facts which will give us some insight on the gradient rank. The following result shows that simultaneous $X$-rank of a set of points can be viewed as the rank of a unique point in a larger ambient space with respect to certain Segre variety $\bbP^s \times X$. The statement already appeared in \cite[Theorem~2.5]{BL:RanksTensorsAndGeneralization} in the setting of border rank. We include the proof to highlight that there is a one-to-one correspondence between sets of points on $X$ providing a simultaneous decomposition for $\calF$ and sets of points on the Segre variety $\bbP^s \times X$ providing a decomposition of the corresponding point. See also \cite[Section~1.3]{Teit:GeometricLowerBoundsGeneralizedRanks}.
\begin{lemma}[{\cite[Theorem 2.5]{BL:RanksTensorsAndGeneralization}}]\label{lemma: simul as segre-veronese}
Let $X \subseteq \bbP W$, $\calF = \{ p_1 \vvirg p_s \} \subseteq \bbP W$, and fix $w_1 \vvirg w_s \in W$ such that $p _i = [w_i] \in \bbP W$. Let $a_1 \vvirg a_s$ be a basis of an $s$-dimensional vector space $A$ and consider $t = \textsum_{i=1}^s a_i \otimes w_i  \in  A \otimes W$. Then
\[
 \rmR_X(\calF) = \rmR_{\nu_{1,1}(\bbP A \times X)} ([t]).
\]
\end{lemma}
\begin{proof}
Suppose $q_1 \vvirg q_r $ are points of $X$ such that $\calF \subseteq \langle q_1 \vvirg q_r \rangle$. Let $z_1 \vvirg z_r \in W$ such that $[z_j] = q_j \in \bbP W$. By definition, $w_i = \sum_{j=1}^r \lambda _{ij} z_j$, for some scalars $\lambda_{ij}$. Thus, we obtain
\begin{align}\label{eq1: lemma 2.5}
t & =  \textsum_{i=1}^s a_i \otimes w_i = \textsum_{\substack{i = 1,\ldots,s \\ j = 1,\ldots,r}} a_i \otimes (\lambda _{ij}z_j) = \nonumber \\
&= \textsum_{\substack{i = 1,\ldots,s \\ j = 1,\ldots,r}} (\lambda_{ij} a_i) \otimes z_j = \textsum_{j=1}^r \left( \textsum_{i = 1}^s \lambda_{ij}a_i \right) \otimes z_j;
\end{align}
hence, $t$ is a linear combination of the $r$ elements $\left( \textsum_i \lambda_{ij}a_i \right) \otimes z_j \in A\otimes W$, for $j=1 \vvirg r$. Taking the corresponding points on $\nu_{1,1}(\bbP A \times X)$, we get $  \rmR_X(\calF) \geq  \rmR_{\nu_{1,1}(\bbP A \times X)} ([t])$.

Conversely, suppose $[t] \in \langle q_1,\ldots,q_r \rangle$, for some $q_j = [y_j \otimes z_j] \in \nu_{1,1}(\bbP A \times X)$. Therefore, in the vector space $A\otimes W$, we have
\begin{equation}\label{eq2: lemma 2.5}
	t =  \textsum_{i=1}^s a_i \otimes w_i = \textsum_{j=1}^r c_j  y_j \otimes z_j, \quad \text{ for some } c_j \in \Bbbk.
\end{equation}
Let $b_1,\ldots,b_s$ be the basis of $A^*$ dual to $a_1,\ldots,a_s$, i.e., $b_k(a_i) = \delta_{ik}$. For every $k = 1 \vvirg s$, apply $b_k$ to both sides of the second equality in \eqref{eq2: lemma 2.5}. Hence, we obtain $w_k = \textsum_{j=1}^r c_j b_k(y_j) z_j$, which expresses every $w_k$ as a linear combination of the $r$ elements $z_1 \vvirg z_r$, with $[z_j] \in X$. This shows $\calF \subseteq \langle [z_1] \vvirg [z_r] \rangle$; thus $\rmR_{\nu_{1,1}(\bbP A \times X)} ( P) \geq \rmR_X(\calF)$. This concludes the proof.
\end{proof}

In the context of gradient rank, we deduce the following.
\begin{corollary}\label{corol: gradient rank as segre-veronese}
 Let $f\in S^d V$. Then, for every $\ud \partinto[m]{d}$ with $d_m = d-k$, we have
 \begin{equation} \label{eqn: corol gradient rank inequality}
 \rmR_{\ud} (f) \geq \gradR{k}(f).
\end{equation}
In particular, 
\begin{equation}\label{eqn: corol gradient rank equality}
\rmR_{(1,d-1)}(f) = \gradR{}(f).
\end{equation}
\end{corollary}
\begin{proof}
Let $\udelta = (d_1 \vvirg d_{m-1})$, so that $\udelta  \partinto k$. By Lemma \ref{lemma: simul as segre-veronese}, we have $\gradR{k}(f) = \rmR_{\nu_{1,1}( \bbP A \times \nu_{d-k}(\bbP V))} ( [t])$ where $\dim A = \dim S^{\udelta} V$ and
\[
t = \sum_{\substack{\underline{\alpha} =  (\alpha_1 \vvirg \alpha_{m-1}) : \\ \vert \alpha_j \vert = d_j }} a_{\underline{\alpha}} \otimes \frac{\partial^k}{\partial \bfx^{\alpha_1 + \cdots \alpha_{m-1}}}f. 
\]
In particular, one can define an isomorphism $A \simeq S^{\udelta} V$ by $a_{\underline{\alpha}} \mapsto c_{\underline{\alpha}}\bfx^{\alpha_1} \ootimes \bfx^{\alpha_{m-1}}$, where the $c_{\underline{\alpha}}$'s are suitable coefficients so that $t$ coincides with $f$ regarded as an element of $S^{\ud} V$ as presented in \eqref{eqn: polarization}.

We have the inclusions 
\begin{equation}\label{eqn: inclusions in lemma grad simul}
\nu_\ud ( \bbP V^{\times m}) \subseteq \nu_{1,d-k}(\bbP S^{\udelta} V \times \bbP V) = \nu_{1,1} (\bbP S^{\udelta} V \times \nu_{d-k}(\bbP V)) ;
\end{equation}
this shows $\rmR_\ud ([f]) \geq \rmR_{\nu_{1,1}(\bbP S^{\udelta} V \times \nu_{d-k}(\bbP V))} ([f])$ and we obtain the inequality \eqref{eqn: corol gradient rank inequality}.

The second statement follows from the fact that when $k=1$, we have $\udelta = (1)$, and therefore the first inclusion in \eqref{eqn: inclusions in lemma grad simul} is an equality. \end{proof}

The non-symmetric analog of the equality in \eqref{eqn: corol gradient rank equality} is a known characterization of tensor rank: given $t \in V_1\otimes\ldots\otimes V_d$, the simultaneous rank of the tensors $\{t(\omega_1) \in V_2\otimes\ldots\otimes V_d ~|~ \omega_1 \in V_1^*\}$ is equal to the tensor rank of $t$ {(see \cite[Theorem 2.1]{Fri13:TensorsBorderRank}, for $d = 3$, or \cite[Exercise 3.1.1.2]{Lan:TensorBook})}. We point out that the equality in \eqref{eqn: corol gradient rank equality} is a consequence of the fact that every element of $\bbP V$ has rank one, because $\nu_1$ is the identity map. When $k \geq 2$, this is no longer true and indeed \eqref{eqn: corol gradient rank inequality} can be a strict inequality, as shown in the following example. 

\begin{example}
Let $\dim V = n+1$. Let $f \in S^3 V$ be any element with $\rmR_3(f) > n+1$, which exists for every $n \geq 1$. Then $\langle \nabla^2 f \rangle \subseteq V$, so $\gradR{2}(f) \leq n+1$, showing $\rmR_3(f) > \gradR{2}(f)$. 
\end{example}

Lemma \ref{lemma: d geq d'} and Corollary \ref{corol: gradient rank as segre-veronese} establish the chain of inequalities in \eqref{eq: intro}.

\subsection{Apolarity theory}\label{subsec: apolarity}

A classical approach to the Waring problem is based on \emph{apolarity theory}, which is the study of the action of the ring of polynomial differential operators on the polynomial ring; see \cite{IarrKan:PowerSumsBook, Gera:InvSysFatPts}. In this section, we recall basic facts on classical apolarity for polynomials and its generalization to (partially symmetric) tensors and simultaneous ranks.

Given a vector space $V$ with basis $\{x_0,\ldots,x_n\}$, let $\{y_0,\ldots,y_n\}$ be its dual basis of~$V^*$. The symmetric algebra $S^\bullet V^*$ can be identified with the algebra of differential operators on $x_0 \vvirg x_n$ with constant coefficients, by identifying $y_j$ with $\frac{\partial}{\partial x_j}$. Hence, for every $i,j$, with $i \leq j$, there is a bilinear map
\begin{align}\label{eq: apolar action}
\circ : S^i V^* \times S^j V & \to S^{j-i} V, \nonumber \\
(\phi, f) & \mapsto \phi \circ f := \phi\left(\textfrac{\partial}{\partial x_0} \vvirg \textfrac{\partial}{\partial x_n}\right) f(\bfx),
\end{align} 
defined by differentiation. In particular, on the monomial basis, for any $\alpha,\beta\in\bbN^{n+1}$ multi-indices with $|\alpha| = j$ and $|\beta| = i$, we have
\[
\bfy^\beta \circ \bfx^\alpha = 
	\begin{cases}
		\frac{\alpha!}{(\alpha-\beta)!}\bfx^{\alpha - \beta} := \prod_{i=0}^n\frac{\alpha_i!}{(\alpha_i-\beta_i)!}  x_i^{\alpha_i - \beta_i} & \text{ if } \beta \leq \alpha \text{, i.e., } \beta_i \leq \alpha_i, \text{ for any } i; \\
		0 & \text{ otherwise}.
	\end{cases}
\] 
Set $S^j V = 0$ whenever $j < 0$ and extend this map via bilinearity to define the {\it apolar action} of $S^\bullet V^*$ on $S^\bullet V$, that we still denote by $\circ$.
\begin{definition}
Given $f \in S^dV$, the {\bf apolar ideal} of $f$ is the ideal in $S^\bullet V^*$ of polynomial differential operators which annihilate $f$, i.e., 
\[
\Ann_{d}(f) := \{\phi \in S^\bullet V^* : \phi \circ f = 0\}.
\]
\end{definition}
The ideal $\Ann_d(f)$ is homogeneous and $(\Ann_d(f) )_i = S^i V^*$, for $i > d$; that is, $\Ann_d(f)$ is Artinian with socle degree $d$. The $i$-th {\it catalecticant} of $f$ is the linear map 
\begin{equation}\label{def:catalecticant}
\begin{aligned}
\cat_i(f) : S^i V^* &\to S^{d-i}V, \nonumber \\ 
\varphi &\mapsto \varphi \circ f. 
\end{aligned}
\end{equation}
Note that $(\Ann_d(f) )_i = \ker \left(\cat_i(f)\right)$, for every $i$.

\begin{remark}
We point out that apolar ideals of homogeneous polynomials are graded Artinian Gorenstein ideals. Moreover, {\it Macaulay's duality} provides a one-to-one correspondence between graded Artinian Gorenstein algebras of socle degree $d$ and homogeneous polynomials of degree $d$; see, e.g., \cite[Theorem 8.7]{Gera:InvSysFatPts} or \cite[Section~21.2]{Eis:CommutativeAlgebra}.
\end{remark}

Together with the interpretation of $S^\bullet V^*$ as ring of polynomial differential operators, we have the natural structure of a ring of polynomials on $V$. In particular, homogeneous ideals in $S^\bullet V^*$ define algebraic varieties and schemes in $\bbP V$. In this way, from the apolar ideal we may obtain Waring decompositions of $f$ as follows.

\begin{lemma}[Apolarity Lemma -- classical version, {\cite[Lemma 1.15]{IarrKan:PowerSumsBook}}]\label{lemma: apolarity homog}
Let $f \in S^dV$. Let $I_\bbX \subseteq S^\bullet V^*$ be the ideal defining a set of points $\bbX = \{[\ell_1],\ldots,[\ell_r]\}\subseteq \bbP V$. Then the following are equivalent:
\begin{enumerate}[{\rm (i)}]
\item $I_ \bbX \subseteq \Ann_d(f)$;
\item $f = \sum_{i=1}^r \lambda_i \ell_i^{d}$, for some $\lambda_i \in \Bbbk$.
\end{enumerate}
If conditions {\rm (i)} and {\rm (ii)} hold, the set $\bbX$ is said to be {\bf apolar} to $f$.
\end{lemma}

Via the Apolarity Lemma, the problem of determining Waring ranks and Waring decompositions of a homogeneous polynomial can be approached by analyzing ideals of sets of points contained in its apolar ideal. 

Note that condition {\rm (ii)} of Apolarity Lemma can be rephrased by saying that, in the same notation as the statement, $[f] \in \langle \nu_d(\bbX)\rangle$. In this form, Apolarity Lemma holds more generally for possibly not reduced $0$-dimensional schemes, see e.g. \cite[Lemma~1]{BJMR:PolynomialsGivenHF}. In particular, if  $\bbX \subseteq \bbP V$ is a $0$-dimensional scheme, then $I_\bbX \subseteq \Ann_d(f)$ if and only if $[f] \in \langle \nu_{d} (\bbX) \rangle$, where the span of a $0$-dimensional scheme is the zero set of the linear forms in its defining ideal.

Moreover, apolarity theory extends to partially symmetric tensors, and even more generally to the context of toric varieties, see e.g. \cite{Gal16:MultiApolarity, GalRanVil:ApolarityToric, Ven18}. For any $\ud \partinto{d}$, the space $S^\ud V$ may be regarded as the multi-homogeneous component of multi-degree $\ud$ in the ring $\Bbbk[x_{ij} : i = 1 \vvirg m , j= 0 \vvirg n]$. In this setting, the apolar action is naturally multigraded and the apolar ideal of $f \in S^\ud V$ is multi-homogeneous; denote it $\Ann_\ud (f)$. Recall that multi-homogeneous ideals define algebraic varieties and schemes in $\bbP V ^{\times m}$. From the toric version of Apolarity Lemma, e.g., \cite[Lemma 1.3]{GalRanVil:ApolarityToric} or \cite[Proposition 3.8]{Gal16:MultiApolarity}, the multi-homogeneous analog of Lemma \ref{lemma: apolarity homog} is as follows.

\begin{lemma}[Apolarity Lemma -- multigraded version]\label{lemma: apolarity multihom}
Let $f \in S^\ud V$. Let $I_\bbX$ be the multi-homogeneous ideal defining a set of points $\bbX = \{ q_1 \vvirg q_r \} \subseteq \bbP V^{\times m}$, with $q_j = ( [\ell_{j,1}] \vvirg [\ell_{j,m}])$. Then the following are equivalent:
\begin{enumerate}[{\rm (i)}]
\item $I_ \bbX \subseteq \Ann_\ud(f)$;
\item $f = \sum_{i=1}^r \lambda_j \ell_{j,1}^{d_1} \ootimes \ell_{j,m}^{d_m}$, for some $\lambda_j \in \Bbbk$.
\end{enumerate}
If conditions {\rm (i)} and {\rm (ii)} hold, the set $\bbX$ is said to be {\bf $\ud$-apolar} to $f$.
\end{lemma}
Again, condition {\rm (ii)} can be stated as $[f] \in \langle \nu_{\ud}(\bbX) \rangle $,  and Lemma \ref{lemma: apolarity multihom} extends to the general case of $0$-dimensional schemes. We use the expression $\ud$-apolar in that setting as well.

It is easy to extend Apolarity Lemma also in the case of simultaneous rank considering sets of points, or more generally $0$-dimensional schemes, which are simultaneously apolar to a set of forms.

\begin{lemma}[Apolarity Lemma -- simultaneous version]\label{remark: simul apolar}
 Let $f_1 \vvirg f_s \in S^d V$. Let $\bbX = \{[\ell_1],\ldots,[\ell_r]\} \subseteq \bbP V$ be a set a points defined by the ideal $I _\bbX$. Then, the following are equivalent:
\begin{enumerate}[{\rm (i)}]
	\item $I_\bbX \subseteq \bigcap_{i=1}^s \Ann_d (f_i) $;
	\item $f_j = \sum_{i=1}^r \lambda_{j,i} \ell_{i}^{d}$, for some $\lambda_{j,i} \in \Bbbk$, for any $j = 1,\ldots,s$.
\end{enumerate}
\end{lemma}
Again, condition (ii) can be stated as $\langle f_1 \vvirg f_s \rangle \subseteq \langle \nu_d ( \bbX ) \rangle$, and Lemma \ref{remark: simul apolar} extends to the case of $0$-dimensional schemes, possibly non reduced. Moreover, Lemma \ref{remark: simul apolar} extends to the case of forms of different degrees.

If $\calF \subseteq S^d V$, $\Ann_{d}(\calF) := \bigcap_{f \in \calF} \Ann_d(f)$ is called the simultaneous apolar ideal of $\calF$. Proposition \ref{prop: apolar gradient} will provide a characterization of $\Ann_{d-k}(\nabla^k f)$ for every $f \in S^d V$ and every $k$. This will be a fundamental tool for the rest of the paper.

\subsection{Hilbert functions of $0$-dimensional schemes}\label{subsec: hilbert func}
Given a homogeneous ideal $I \subseteq S^\bullet~V^*$, the ideal $I$ and the quotient algebra $A_I := {S^\bullet V^*}/{I}$ inherit the grading of the polynomial ring. The {\it Hilbert function} of the quotient algebra $A_I$ is the function which sends an integer $i$ to the dimension, as a $\Bbbk$-vector space, of the component of degree $i$ of $A_I$, i.e.,
\begin{equation}\label{eq: HF}
\HF(A_I; i) := \dim_\Bbbk (A_I)_i = \dim_\Bbbk S^iV^* - \dim_\Bbbk I_i. 
\end{equation}
If $I = I_\bbX$ is the defining ideal of a $0$-dimensional scheme $\bbX \subseteq \bbP V$, $\HF_\bbX := \HF(A_{I_\bbX}; -)$ denotes the Hilbert function of the corresponding graded algebra. 

{The Hilbert function of a $0$-dimensional scheme $\bbX$ is strictly increasing until it reaches a constant value; the value of this constant is the \emph{degree} of $\bbX$, denoted $\deg(\bbX)$. See \cite[Theorem 1.69]{IarrKan:PowerSumsBook}. In the case where $\bbX$ consists of $s$ simple points, then $\deg(\bbX) = s$. We refer to \cite{EisHar:GeometrySchemes} for the theory of $0$-dimensional schemes.}

\begin{lemma}\label{lemma: basics HF}
Let $I_\bbY \subseteq S^\bullet V^*$ be an ideal defining a $0$-dimensional scheme $\bbY \subseteq \bbP V$. Let $H = \{\ell = 0\}\subseteq\bbP V$ be a hyperplane such that $\bbY \cap H = \emptyset$. Then
\[
\deg(\bbY) = \sum_{i \geq 0} \HF(A_{I_{\bbY}+(\ell)};i).
\]
\end{lemma}
\begin{proof}
By assumption, $\ell$ is a non-zero divisor in $A_{I_\bbY}$. Therefore, for every $i$, multiplication by $\ell$ induces the exact sequence
\begin{equation}\label{eqn: exact sequence}
	0 \longrightarrow \left(A_{I_{\bbY}}\right)_{i-1} \overset{\cdot \ell}{\longrightarrow} \left(A_{I_{\bbY}}\right)_{i} \longrightarrow \left(A_{I_{\bbY} + (\ell)}\right)_{i} \longrightarrow 0.
\end{equation}
This gives $\HF \left(A_{I_\bbY + (\ell)} , i \right) = \HF_{\bbY}(i) - \HF_\bbY(i-1)$, namely the Hilbert function of $A_{I_\bbY + (\ell)}$ is the {\it first difference} of the Hilbert function of ${\bbY}$. Hence, for $s \gg 0$, we have $ \deg(\bbY) = \HF_{\bbY}(s) = \sum_{i = 0}^s \big(\HF_{\bbY} (i) - \HF_{\bbY}(i-1)\big) $, where $\HF_{\bbY}(s)$ is written as a telescopic sum, using $(S^\bullet V^* / I_{\bbY})_j = 0$ if $j < 0$. We conclude
\[
	\deg(\bbY) = \sum_{i \geq 0}\HF(A_{I_{\bbY} + (\ell)};i).  
\]
\end{proof}
The latter results justify the following definition.
\begin{definition}\label{def: regularity}
Let $\bbY\subseteq \bbP V$ be a $0$-dimensional scheme. We call the first degree where the Hilbert function of $\bbY$ stabilizes the {\bf regularity index} of $\bbY$, denoted $\reg(\bbY)$, i.e., 
$${\reg}(\bbY) := \min\{i : \HF_\bbY(i) = \deg(\bbY)\}.$$
\end{definition}

\begin{remark}\label{rmk: regularity}
By \cite[Theorem 1.69]{IarrKan:PowerSumsBook}, the regularity index of a $0$-dimensional scheme is one less than the {\it Castelnuovo-Mumford regularity} of its defining ideal; for details, we refer to \cite[Section~20.5]{Eis:CommutativeAlgebra}. We recall that the Castelnuovo-Mumford regularity of an ideal bounds from above the maximal degree of any minimal set of generators of the ideal: in particular, 
\[
	\max\left\{\deg(f_i) : \begin{array}{c} \text{for any minimal set of generators} \\ I_\bbY = (f_1,\ldots,f_s)\end{array} \right\} \leq {\rm reg}(\bbY) + 1.
\]
\end{remark}

\begin{remark}\label{rmk: multigraded HF}
If $I \subseteq S^\bullet V^* \ootimes S^\bullet V^*$ is a multi-homogeneous ideal, then again it inherits the multi-grading of the ring and we define the {\it multigraded Hilbert function} of the corresponding quotient algebra analogously to \eqref{eq: HF} by considering any multi-degree. From \cite[Proposition~1.9]{SiVT06:MultigradedRegularity}, we observe the following. If $\bbY$ is a $0$-dimensional scheme in $\bbP V^{\times m}$, the multigraded Hilbert function of $\bbY$ is increasing and eventually constant in each direction, that is, for~any~$i \in \{1,\ldots,m\}$:
\begin{enumerate}[\rm (i)]
\item $\HF_\bbY(\ud) \leq \HF_\bbY(\ud + \ue_i)$, for any $\ud \in \bbN^m$;
\item $\HF_\bbY(\ud) = \HF_\bbY(\ud+\ue_i)$, then $\HF_\bbY(\ud) = \HF_\bbY(\ud+2\ue_i)$, for any $\ud \in \bbN^m$.
\end{enumerate}
Moreover, $\HF_\bbY(\ud) \leq \deg(\bbY)$, for any $\ud \in \bbN^m$, and equality holds if $d_i \gg 0$, for all $i$.
\end{remark}

\subsection{Cactus ranks}\label{subsec: cactus}

Considering arbitrary $0$-dimensional schemes suggests the definition of a more general notion of rank: the \textit{cactus rank}. This was introduced in \cite{IarrKan:PowerSumsBook} in the setting of homogeneous polynomials with the name of \emph{scheme length}. The terminology cactus rank, which is now the one commonly used in the literature, was introduced in \cite{RanSch11:RankForm, BerRan:CactusRankCubicForm,BuczBucz:SecantVarsHighDegVeroneseReembeddingsCataMatAndGorSchemes}.

\begin{definition}\label{defin: cactus rank}
 Let $X \subseteq \bbP^N$ be a non-degenerate projective variety and let $p \in \bbP^N$. The $X$-{\bf cactus rank} of $p$ is 
 \[
\cactusR_X(p) := \min \left\{ r : \begin{array}{ll}
                                      \text{there exists a $0$-dimensional scheme $\bbY \subseteq X$ } \\
                                      \text{with $p \in \langle \bbY \rangle$ and $\deg(\bbY) =r$}
                                    \end{array} \right\} .  
 \]
\end{definition}

By Apolarity Lemma, we can characterize the cactus rank with respect to Segre-Veronese varieties as follows. Let $f \in S^\ud V$. Then the cactus rank of $[f]$ with respect to the $\ud$-th Segre-Veronese variety is
\[
 \cactusR_\ud (f) = \min \left\{ r : \begin{array}{ll}
                                      \text{there exists a $0$-dimensional scheme $\bbX\subseteq\bbP V^{\times m}$} \\
                                      \text{$\ud$-apolar to $f$ with $\deg(\bbX) = r$}
                                    \end{array} \right\} .
\]

The analog of Lemma \ref{lemma: d geq d'}, with the same proof, holds for cactus rank as well.
\begin{lemma}\label{lemma: cactus d geq d'}
Let $f\in S^d V$ and $\ud \succeq \ud'$. Then $\cactusR_\ud(f)\geq \cactusR_{\ud'}(f)$. 
\end{lemma}

Also simultaneous rank has a corresponding cactus version.

\begin{definition}\label{defin: simul cactus}
Let $X \subseteq \bbP^N$ be a non-degenerate projective variety and let $\calF \subseteq \bbP ^N$. The {\bf simultaneous $X$-cactus rank} of $\calF$, denoted $\cactusR_X(\calF)$, is the minimum $r$ such that there exist $\bbY \subseteq X$ with $\deg(\bbY) = r$ and $\calF \subseteq \langle \bbY \rangle$, or equivalently $\langle \calF \rangle \subseteq \langle \bbY \rangle$.
\end{definition}

However, a cactus analog of Lemma \ref{lemma: simul as segre-veronese} fails, as shown in Example \ref{example: non simul SV for cactus}. 

As in the case of simultaneous rank, we are interested in relations between $\cR_\ud(f)$ and the simultaneous rank of partial derivatives of $f$.
\begin{definition}
Let $f \in S^d V$ and let $k < d$ be a positive integer. The {\bf $k$-th gradient cactus rank} of $f$ is the simultaneous cactus rank of $\nabla^k f$ with respect to the $(d-k)$-th Veronese variety; write 
		$$
			\gradc{k}(f) := \cR_{\nu_{d-k}\bbP V}(\nabla^k f).
		$$
\end{definition}

The following result gives a partial analog of Lemma \ref{lemma: simul as segre-veronese} in the case of cactus rank. 

\begin{lemma}\label{lemma: cactus simul less than segre-veronese}
Let $X \subseteq \bbP W$, $\calF = \{ p_1 \vvirg p_s \} \subseteq \bbP W$, and fix $w_1 \vvirg w_s \in W$ such that $p _i = [w_i] \in \bbP W$. Let $a_1 \vvirg a_s$ be a basis of an $s$-dimensional vector space $A$ and consider $t = \textsum_{i=1}^s a_i \otimes w_i  \in  A \otimes W$. Then
\[
 \cR_X(\calF) \leq \cR_{\nu_{1,1}(\bbP A \times X)} ([t]).
\]
\end{lemma}
\begin{proof}
 Let $\pi: \bbP A \times \bbP W \to \bbP W$ be the projection onto the second factor. Let $\bbX \subseteq \bbP A \times X$ be a $0$-dimensional scheme such that $[t] \in \langle \nu_{1,1}(\bbX) \rangle \subseteq \bbP (A\otimes W)$. Let $\bbY = \pi(\bbX) \subseteq X$. Then $\deg(\bbY) \leq \deg(\bbX) = \deg(\nu_{1,1}(\bbX))$. We will show $\calF \subseteq \langle \bbY \rangle$.

It suffices to show that $t \in A \otimes E$, where $E$ is defined by $\bbP E := \langle \bbY \rangle \subseteq \bbP W$. Indeed, $t \in A \otimes E$ implies that the image of the linear map $t : A^* \to W$ is contained in $E$, namely $\langle w_1 \vvirg w_s \rangle \subseteq E$. In particular $\calF \subseteq \bbP E$.

We have $\bbX \subseteq \pi^{-1}(\bbY) \subseteq \pi^{-1}(\bbP E) = \bbP A \times \bbP E$. Applying $\nu_{1,1}$ and passing to the linear spans, $\langle \nu_{1,1}(\bbX) \rangle \subseteq \langle \nu_{1,1}(\bbP A \times \bbP E) \rangle = \bbP( A \otimes E)$. Since $[t] \in \langle \nu_{1,1}(\bbX) \rangle$ we obtain $t \in A \otimes E$ and we conclude.
\end{proof}

A consequence of Lemma \ref{lemma: cactus simul less than segre-veronese} is the cactus analog of Corollary \ref{corol: gradient rank as segre-veronese}. In particular, Example \ref{example: non simul SV for cactus} shows that the analog of \eqref{eqn: corol gradient rank equality} does not hold for cactus rank.

\begin{corollary}\label{corol: cactus gradient rank as segre-veronese}
 Let $f\in S^d V$. Then, for every $\ud \partinto[m]{d}$ with $d_m = d-k$, we have
 \begin{equation*} 
 \cR_{\ud} (f) \geq \gradc{k}(f).
\end{equation*}
\end{corollary}
\begin{proof}
This follows by the same argument of Corollary  \ref{corol: gradient rank as segre-veronese}, using Lemma \ref{lemma: cactus simul less than segre-veronese} instead of Lemma \ref{lemma: simul as segre-veronese}.
\end{proof}

The results of  Lemma \ref{remark: simul apolar}, Proposition \ref{prop: apolar gradient} and Corollary \ref{corol: cactus gradient rank as segre-veronese} provide the following chain of inequalities, which is the cactus version of \eqref{eq: intro}. For every $\ud \partinto{d}$ with $d_m = d-k$,
\begin{equation}\label{ineq: cactus ranks}
\cR_d(f) \geq \cR_\ud(f) \geq \gradc{k}(f). 
\end{equation}

We conclude this section by providing some insights on the relations between the ranks of the catalecticant maps, and more generalized flattening maps, and the (partially symmetric) rank of a form $f$. See also \cite[Example 4.7]{BGL:DeterminantalEquationsEKSConj}

\begin{remark}\label{rmk: flattenings}
By the the Apolarity Lemma (Lemma \ref{lemma: apolarity homog}), 
\[
 \HF_\bbX (i) \geq \HF(A_{\Ann_d(f)};i) = \rank (\cat_i(f)), \quad \text{for any } i \in \bbN, 
\]
for every $f \in S^d V$, and every $\bbX$ apolar to $f$. In particular 
\[
\rmR_d(f) \geq \cR_d(f) \geq \max_{i = 0,\ldots,d} \{\HF(A_{\Ann_d(f)};i)\}. 
\]
The maximal value of the Hilbert function of the quotient algebra of $\Ann_d(f)$ is sometimes referred to as \emph{catalecticant lower bound} for $\cR_d(f)$. Similar inequalities hold for the partially symmetric rank by considering the {\it multigraded} Hilbert function.

We observe that catalecticant lower bounds hold for $\cR_{(1^d)}(f)$, where $(1^d) = (1,\ldots,1)$. More precisely, for every $t \in {V_1\ootimes V_d}$ and every subset $I \subseteq \{ 1 \vvirg d\}$, there is an induced linear map, called \emph{flattening map}, 
\[
\flat_I(t) : \bigotimes_{i \in I} V_i^* \to \bigotimes_{i \in I^c} V_i,\quad \text{ where } I^c = \{1,\ldots,d\} \smallsetminus I, 
\]
defined by contraction; see, e.g., \cite[Chapter 2]{Lan:TensorBook}. Its rank is the value of the Hilbert function of $A_{\Ann_{1^d}(t)}$ in multi-degree $\underline{e}_i$; notice that $A_{\Ann_{(1^d)}(t)}$ is a quotient of the ring $S^\bullet V_1 \ootimes S^\bullet V_d = S^\bullet (V_1 \ooplus V_d)$. This rank is a lower bound for $\cR_{1^d}(t)$. Now, if $f \in S^d V$, then one has that $\rank (\cat_i(f)) = \rank (\flat_I(f))$, for every $I$ with $\vert I \vert =i$. In conclusion, the catalecticant lower bound is indeed a lower bound for $\cR_{1^d}(f)$. 

More generally, some generalized flattening maps for $f \in S^d V$, naturally providing lower bounds for $\rmR_d(f)$, give lower bounds for $\rmR_{1^d}(f)$ and, by \cite{Galazka:VectorBundlesGiveEqnsForCactus}, for $\cR_{1^d}(f)$ as well. We observe this fact for Koszul flattenings  \cite{LanOtt:EqnsSecantVarsVeroneseandOthers}: given $f \in S^d V$, define $\flatKos_i^{\wedge p} (f):  S^i V^* \otimes \Lambda^p V  \to S^{d-i-1} \otimes \Lambda^{p+1} V$ to be the composition
\[
 S^i V^* \otimes \Lambda^p V \xrightarrow{\cat_i(f) \otimes \id_{\Lambda^p V}} S^{d-i} V \otimes \Lambda^p V \longrightarrow S^{d-i-1}V \otimes \Lambda^{p+1} V,
\]
where the second map is the Koszul differential. From \cite[Proposition 4.1.1]{LanOtt:EqnsSecantVarsVeroneseandOthers}, $$\rmR_{d} (f) \geq \frac{\rank (\flatKos_i^{\wedge p} (f))  }{ \rank (\flatKos_i^{\wedge p}(x_0^{\otimes d}) )} =  \frac{\rank (\flatKos_i^{\wedge p} (f))  }{\binom{n}{p-1}},$$ if $\dim V = n+1$. In the non-symmetric setting, for $t \in V^{\otimes d}$, one defines a Koszul flattening in a similar way, as an \emph{augmentation} of $\flat_I$:
\[
 \flatKos_I^{j,\wedge p } (t) : \textbigotimes_I V^* \otimes \Lambda^p V \xrightarrow{\flat_I \otimes \id_{\Lambda^p V}} \textbigotimes_{I^c} V \otimes \Lambda^p V \longrightarrow \textbigotimes_{I^c \setminus \{ j\} } V \otimes \Lambda ^{p+1} V.
\]
This provides the lower bound $$\rmR_{1^d} (t) \geq \frac{\rank (\flatKos_I^{j,\wedge p} (t))  }{ \rank (\flatKos_I^{j,\wedge p}(x_0^{\otimes d}))} =  \frac{\rank (\flatKos_I^{j,\wedge p} (t))  }{\binom{n}{p-1}}.$$ Analogously to the standard flattening case, for $f \in S^d V$, one obtains that $$\rank (\flatKos_I^{j,\wedge p} (f)) = \rank (\flatKos_i^{\wedge p} (f)),$$ for every set of indices $I$ with $\vert I \vert = i$; therefore the Koszul flattening lower bound for $f$ holds for $\rmR_{1^d}(f)$. By the results of \cite{Galazka:VectorBundlesGiveEqnsForCactus}, these bounds hold for cactus rank as well. 
\end{remark}

\subsection{Consequences of apolarity theory}
In this section, we provide some immediate consequences of the theory introduced in Section \ref{subsec: ranks alg vars} and Section \ref{subsec: hilbert func}. The main result of this section is Proposition \ref{prop: apolar gradient} which gives an explicit description of the simultaneous apolar ideal of the set of partial derivatives of a given order of a form $f$.

\subsubsection{Failure of {Lemma \ref{lemma: simul as segre-veronese}} for cactus rank}
First, we provide an example showing that the simultaneous cactus rank of a family of forms cannot be read as the cactus rank of tensor in an bigger space unlike what happens for the classical rank in Lemma \ref{lemma: simul as segre-veronese}.

\begin{example}\label{example: non simul SV for cactus}
Consider $\calF = \{x_0^2 x_1 , x_0^2 x_2 \} \subseteq S^3 V$ with $\dim V = 3$. The apolar ideal is $\Ann_3(\calF) = (y_0^3,y_1^2,y_2^2,y_1y_2)$ whose Hilbert function in degree $2$ is equal to $3$ and, therefore, any $0$-dimensional scheme apolar to $\calF$ has length at least $3$. Indeed, $\cR_3(\calF) = 3$ since the \textit{$2$-fat point} supported at $[x_0] \in \bbP V$, i.e., the $0$-dimensional scheme of degree $3$ defined by $(y_1,y_2)^2$, is apolar to $\calF$.

Now, consider the partially symmetric tensor $t = a_0 \otimes  x_0^2 x_1 + a_1 \otimes  x_0^2 x_2  \in A \otimes S^3 V$ with $\dim A = 2$, $A = \langle a_0,a_1\rangle$. We prove that $\cR_{(1,3)}(t) \geq 4$. The bi-graded apolar ideal of $t$ is 
\[
\Ann _{(1,3)} (t) = (S^2 A^*) + (b_0y_2,b_1y_1,b_0y_1-b_1y_2,y_2^2,y_1y_2,y_1^2,y_0^3) \subseteq S^\bullet A^* \otimes S^\bullet V^*,
\]
where $\{b_0,b_1\}$ and $\{y_0,y_1,y_2\}$ are the bases of $A^*$ and $V^*$ dual to $\{a_0,a_1\}$ and $\{x_0,x_1,x_2\}$, respectively. The bi-graded Hilbert function of the quotient algebra $A_{\Ann_{(1,3)}(t)}$ is
\[
\begin{array}{c|ccccc}
\bsfrac{V^*}{A^*} & ${\footnotesize{0}}$ & ${\footnotesize{1}}$ & ${\footnotesize{2}}$ & ${\footnotesize{3}}$ & ${\footnotesize{4}}$ \\ \midrule 
${\footnotesize{0}}$ & 1 & 3 & 3 & 2 & - \\
${\footnotesize{1}}$ & 2 & 3 & 3 & 1 & - \\
${\footnotesize{2}}$ & - & - & - & - & - \\
 \end{array} 
\]
Therefore, $\cR_{(1,3)}(t) \geq 3$. Assume that there exists a $0$-dimensional scheme $\bbY \subseteq \bbP A \times \bbP V$ of length $3$ apolar to $t$. Then, 
\[
(I_\bbY)_{(1,2)} = \left( (\Ann_{(1,3)}(t))_{(1,2)}\right) = (b_0y_2,b_1y_1,b_0y_1-b_1y_2,y_2^2,y_1y_2,y_1^2) =: J .
\]
In particular, $I_\bbY \supseteq J$. Since $I_\bbY$ is the ideal of a $0$-dimensional scheme in $\bbP A \times \bbP V$, it is saturated with respect to the bigraded irrelevant ideal $\frakm = (b_0,b_1) \cdot (y_0,y_1,y_2)$. Let $J^{sat} = J : \frakm^\infty$ be the saturation of $J$. We deduce $I_\bbY \supseteq J^{\it sat} = (y_1,y_2)$. This shows $\bbY \subseteq \bbP A \times (1:0:0)$. This is a contradiction because $\nu_{(1,3)}(\bbY) \subseteq \nu_{(1,3)}(\bbP A \times (1:0:0))$ and the latter is a line not containing the tensor $t$.
\end{example}

\subsubsection{Structure of simultaneous apolar ideal}
The next result computes the simultaneous apolar ideal of $\nabla^k f$ for a given $f \in S^d V$ which, via apolarity theory, will be of key importance for our computations.

\begin{proposition}\label{prop: apolar gradient}
Let $f \in S^d V$ and let $k \geq 0$. For every $i \geq 0$,
\[
\left(\Ann_{d-k} (  \nabla^k f)\right)_i = \bigcap_{\vert \alpha \vert = k} \left(\Ann_{d-k} \left(\textfrac{\partial^k}{\partial \bfx^\alpha} f \right)\right)_i = \left\{ \begin{array}{ll}
                                         (\Ann_{d} ( f))_i & \text{if $0 \leq i \leq d-k$}, \\
                                         S^i V^* & \text{if $i \geq d-k+1$}.
                                        \end{array}
\right.
\]
\end{proposition}
\begin{proof}
For $i \geq d-k+1$ the statement follows simply because $k$-th partial derivatives of $f$ have degree $d-k$.

For $i \leq d-k$, the statement is a consequence of the fact that differential operators commutes. For every $\phi \in S^i V^*$, we have 
\begin{equation}\label{eqn: apolar gradient prop}
\phi \circ \textfrac{\partial^\alpha}  {\partial ^k \bfx} f = \phi \circ \left(\bfy^\alpha \circ f\right) =  \bfy^\alpha \circ (\phi \circ  f).
\end{equation}
If $\phi \in (\Ann_{d} ( f))_i$, the right-hand-side of \eqref{eqn: apolar gradient prop} is $0$, showing that the left-hand-side is $0$ for every $\alpha$, and therefore $\phi \in (\Ann_{d-k} (  \nabla^k f))_i$. Conversely, if $\phi \in (\Ann_{d-k} (  \nabla^k f))_i$, then $\phi \in  \Ann_{d-k} \left({\textfrac{\partial^k}{\partial \bfx^\alpha}} f \right)$ for every $\alpha$; therefore the left-hand-side of \eqref{eqn: apolar gradient prop} is $0$, which implies that the right hand side is $0$; in this case we deduce that ${\phi \circ f}$ is a homogeneous polynomial of degree $k$ which is annihilated by all differential operators of order $k$. Since the apolarity pairing is non-degenerate, we conclude that $\phi \circ f = 0$. 
\end{proof}

An immediate consequence of Proposition \ref{prop: apolar gradient} is the following fact:

\begin{remark}\label{rmk: apolar to f iff apolar to nablaf}
Let $f \in S^d V$ and $k \geq 0$ as in Proposition \ref{prop: apolar gradient}. Let $\bbX \subseteq \bbP V$ be a $0$-dimensional scheme such that $I_\bbX$ is generated in degree at most $d-k$. Then $\bbX$ is apolar to $f$ if and only if it is apolar to $\nabla^k f$.
\end{remark}

 \subsubsection{Sylvester's Theorem for binary forms}\label{sec: Sylvester}
As a first explicit example of application of apolarity theory to compute ranks of homogeneous polynomials, we recall Sylvester's Theorem which completely describes the Waring decompositions in the case of binary forms \cite{Sylvester1852}.

Let $\dim V = 2$. One can prove that if $f \in S^d V$, then $\Ann_d(f) = (g_1 , g_2)$ where $\deg(g_i) = e_i$ and $e_1 + e_2 = d+2$; this is a consequence of the general theory, and more precisely of the fact that Gorenstein algebras of codimension $2$ are always complete intersection and that Artinian Gorenstein algebras have symmetric Hilbert function, i.e. for $i = 0,\ldots,d$, $\HF(A_{\Ann_d(f)};i) = \HF(A_{\Ann_d(f)};d-i)$; see \cite[Proposition 8.6]{Gera:InvSysFatPts}. Hence, let $e _1 \leq e_2$.

Recall that $0$-dimensional schemes in $\bbP^1$ are defined by principal ideals. Hence, if $g_1$ has distinct roots, we conclude that $\rk_{d}(f) = e_1$ and a minimal set of points apolar to $f$ is given by the roots of $g_1$; moreover, if $e_1 < e_2$, this is the unique minimal set of points apolar to $f$. If $g_1$ does not have distinct roots, then a minimal set of points apolar to $f$ is given by the roots of $g_1h+g_2$, for a generic choice of $h \in S^{e_2-e_1}V$. For an exposition of Sylvester's Theorem in modern terminology we refer to \cite{ComSei11:RankBinary}.

\begin{theorem}[Sylvester's Theorem]\label{thm: sylvester}
Let $f \in S^dV$ with $\dim_\Bbbk V = 2$. Let $\Ann_d(f) = (g_1,g_2)$ with $\deg(g_1) \leq \deg(g_2)$. Then
\[
	\rk_{d}(f) = 
	\begin{cases}
			\deg(g_1) & \text{ if $g_1$ has distinct roots}; \\
			\deg(g_2) & \text{ otherwise}.
		\end{cases}
\]
\end{theorem}

As for cactus rank, with the same notation as above, one has $\cactusR_d(f) = e_1$. Indeed, $(g_1)$ always defines a $0$-dimensional scheme of degree $e_1$ apolar to $f$ and there are no apolar schemes of smaller degree since there are no elements of smaller degree in the apolar ideal of $f$. If $e_1 < e_2$, then the $0$-dimensional scheme defined by $g_1$ is the unique minimal $0$-dimensional scheme apolar to $f$.

\section{Computations}\label{sec: computations}
In this section, we prove our main results. We consider special families of symmetric tensors and we study their $k$-th gradient (cactus) ranks. As we already explained, we will focus mostly on the cases where the inequalities of \eqref{eq: intro} become equalities. 

\subsection{Binary forms}
In this section, we obtain a complete result on the gradient ranks and gradient cactus ranks of binary forms. 

\begin{proposition}\label{prop: binary forms}
Let $d\in \mathbb N$ and $f \in S^d V$, with $\dim V  = 2$. Then, for any $k < d$,
\[
 \gradR{k}(f) = \min \{ \rmR_d(f) , d-k+1 \} \quad \text{ and }  \quad \gradc{k}(f) = \min \{ \cactusR_d(f) , d-k+1 \}. 
\]
Consequently, for any $\ud \partinto[m]{d}$ with $d_m = d-k$, we have:
\begin{enumerate}[\rm (i)]
	\item if $\rmR_d(f) \leq d-k+1$, then $\rk_d(f) = \rk_\ud(f) = \gradR{k}(f)$;
	\item if $\cactusR_d(f) \leq d-k+1$, then $\cactusR_d(f) = \cactusR_\ud(f) = \gradc{k}(f)$.
\end{enumerate}
 \end{proposition}
\begin{proof}

By definition, $\gradR{k}(f) \leq \rmR_d(f)$ and $\gradc{k}(f) \leq \cactusR_d(f)$. For $\rmR_d(f) < d-k+1$ (respectively, $\cactusR_d(f) < d-k+1$), we conclude by Remark \ref{rmk: apolar to f iff apolar to nablaf}. Conversely, suppose $\rmR_d(f) \geq d-k+1$ (respectively, $\cactusR_d(f) \geq d-k+1$). Since $(\Ann_{d-k}(\nabla^k f)) _{d-k+1} = S^{d-k+1} V^*$ by Proposition \ref{prop: apolar gradient}, any square-free element (respectively, any element) of $ S^{d-k+1} V^*$ defines a set of $d-k+1$ points (respectively, a $0$-dimensional scheme of degree $d-k+1$) in $\bbP V$ apolar to $f$. This implies $\gradR{k}(f) \leq d-k+1$ (respectively, $\gradc{k}(f) \leq d-k+1$). Again, the lower bound follows by Remark \ref{rmk: apolar to f iff apolar to nablaf}. The second part of the statement follows from the first one by the chain of inequalities \eqref{eq: intro}.
\end{proof}

\begin{remark}
In \cite[Corollary 3.12]{ZhangHuangQi:ComonsConjecture}, the authors proved that the original Comon's question (Question \ref{question} for $\ud = 1^d \partinto{d}$) has an affirmative answer in the case of binary forms. Therefore, by Lemma \ref{lemma: d geq d'}, it follows that Question \ref{question} has an affirmative answer for any $\ud \partinto{d}$. In fact, by \eqref{eqn: corol gradient rank equality} in Corollary  \ref{corol: gradient rank as segre-veronese}, this implies the part (i) of the statement in Proposition \ref{prop: binary forms} in the case $k = 1$.
\end{remark}

\begin{remark}\label{rmk: simul decomp binary}
The proof of Theorem \ref{thm: gradient monomials} gives interesting insights on minimal schemes apolar to the $k$-th gradient of a binary form and, in particular, on their relations with minimal schemes apolar to the form itself. Here, we resume some observations:
\begin{enumerate}[\rm (i)]
	\item if $\rmR_d(f) < d-k+1$ (respectively, $\cactusR_d(f) < d-k+1$), the minimal reduced (respectively, not necessarily reduced) $0$-dimensional schemes apolar to $f$ are the same as the ones minimally simultaneously spanning $\nabla^k f$. Note that for $\rmR_d(f) < \frac{d+1}{2}$ (respectively, $\cactusR_d(f) < \frac{d+1}{2}$), such a reduced (respectively, not necessarily reduced) $0$-dimensional scheme is unique by Sylvester's Theorem \ref{thm: sylvester} (respectively, for the comments on cactus ranks of binary forms at the end of Section \ref{sec: Sylvester});
	\item if $\rmR_d(f) = d-k+1$, we have that the rank $\rmR_d(f)$ and the gradient rank $\gradR{k}(f)$ are the same, but we can find minimal schemes apolar to $\nabla^k f$ which are not apolar to $f$ itself. For example, $x_0x_1^{d-1}$ has rank $d$ and any minimal apolar set of $d$ points does not involve the point $[x_1] \in \bbP V$, see \cite[Section~3.2]{CCO17:WaringLoci}. However, if we consider the first partial derivatives $\nabla f = \{x_1^{d-1}, x_0x_1\}$, we have that the set of points $\bbX = \{[x_1]\} \cup \{[x_0+\xi x_1]: \xi^{d-1} = 1\}$ are apolar to $\nabla f$; indeed, $I_\bbX = \big(y_0(y_0^{d-1}-y_1^{d-1})\big)$ which is contained in $\Ann_{d-1}(\nabla f) = (y_0^2,y_1^d,y_0y_1^{d-1})$.
	\item More generally, if $ \rmR_d(f) \geq d-k+1$ (respectively, $\cactusR_d(f) \geq d-k+1$) then $\gradR{k}(f) = d-k+1$ (respectively, $\cactusR_d(f) = d-k+1$):  in such a case, any set of $d-k+1$ points (respectively, any $0$-dimensional scheme of degree $d-k+1$) is apolar to $\nabla^k f$. Indeed, such a scheme is defined by a principal ideal whose generator has degree $d-k+1$ and, therefore, it is contained in the apolar ideal of the $k$-th gradient of $f$ because, by Proposition \ref{prop: apolar gradient}, $\left(\Ann_{d-k}(\nabla^k f)\right)_{d-k+1} = S^{d-k+1}V^*$.
\end{enumerate}
\end{remark}

\subsection{Ternary and quaternary cubics}\label{sec: cubics}

Comon's question in the case of cubic forms in three or four variables, that is $f \in S^3 V$ with $\dim V = 3,4$ has an affirmative answer: the proof exploits the fact that in these two cases it is possible to classify the orbits under the action of the group $GL(V)$: see \cite[Theorem 7.1(4)]{Fri:RmksSymmetricRank} for three variables and \cite[Theorem 1.3]{Seig:RanksSymRanksCubSurf} for four variables. The statement for cactus rank follows from \cite[Lemma 3.1 and Theorem 1.5]{Seig:RanksSymRanksCubSurf}, which prove that the Question \ref{question} has affirmative answer for border rank in the case of ternary and quaternary cubics, respectively, and from \cite[Sections 3.5 and 3.6]{BuBu15:OnDifference}, which guarantee that in these cases border rank coincides with cactus rank.

Proposition \ref{prop: gradient plane cubics} uses the techniques developed in Section \ref{section: basics} to recover the result in three variables and proves additionally the equality $\cactusR_{1,2}(f) = \gradc{}(f)$. Proposition \ref{prop: gradient cubic surfaces} proves the equality $\cactusR_{1,2}(f) = \gradc{}(f)$ in the case with four variables.

\begin{proposition}\label{prop: gradient plane cubics}
	Let $f \in S^3V$, with $\dim V = 3$. Then
\begin{align*}
		&\rk(f) = \rk_{1,2}(f) = \gradR{}(f) \quad \text{ and } \\
		&\cactusR(f) = \cactusR_{1,2}(f) = \gradc{}(f).  
\end{align*}
\end{proposition}
\begin{proof}
If the first catalecticant of $f$ is not full-rank, then there is a choice of coordinates such that $f$ can be written in fewer variables; in this case $f$ is a binary form and the statement follows from Proposition \ref{prop: binary forms}.

Hence, assume that the first catalecticant of $f$ is full-rank, which implies that rank and cactus rank of $f$ are at least $3$. Therefore, if $f$ has rank $3$ (cactus rank $3$, respectively), then the claim directly follows.

Let $f$ have rank $4$ (cactus rank $4$, respectively) and suppose that $\nabla f$ has an apolar reduced (not necessarily reduced, respectively) $0$-dimensional scheme $\bbX$ with $\deg(\bbX) = 3$. By Proposition \ref{prop: apolar gradient}, we have that $\HF(A_{\Ann_{2}(\nabla f)};1) = 3$, which implies that $\bbX$ is not contained in a line. Since $\deg(\bbX) = 3$, the ideal $I_{\bbX}$ is generated by three quadrics, so $I_\bbX$ is generated in degree $2$. By Proposition \ref{prop: apolar gradient}, we deduce $I_\bbX \subseteq \Ann_3(f)$, contradicting the assumption that $f$ has rank (cactus rank, respectively) $4$.

The cactus rank of plane cubics is at most $4$, see e.g. \cite[Section~3.5]{BuBu15:OnDifference}, so the second part of the statement is proved. The rank of plane cubics is at most $5$ and there is a unique form of rank $5$ up to change of coordinates, which is $f = x_0(x_0x_1+x_2^2)$; see for instance \cite{LaTe10:RanksBorderRanks}. Suppose $\gradR{}(f) \leq 4$ and let $\bbX$ be a set of four points apolar to $\nabla f$. The set $\bbX \subseteq \bbP V = \bbP^2$ may have two possible configurations: either the points in $\bbX$ are in general linear position or three of them lie on a line $\ell$. In the first case, two conics generate $I_{\bbX}$ so $I_{\bbX}\subseteq \Ann_3(f)$ which contradicts that $\gradR{}(f) = 4$. In the second case, one can easily show that $I_{\bbX}$ cannot be radical, which is a contradiction.
\end{proof}
 
Even though the ranks coincide, simultaneous decompositions of the gradient of a plane cubic do not always come from decompositions of the cubic itself. Indeed, as already observed in the case of binary forms (Remark \ref{rmk: simul decomp binary}(ii)), sometimes it is possible to construct a simultaneous decomposition of the gradient which contains some of the forbidden points (in the sense of \cite{CCO17:WaringLoci}) of the original form.

\begin{example}\label{ex: maximal rank cubic}
Let $f = x_0(x_0x_1 + x_2^2)$ be the unique plane cubic of maximal rank up to choice of coordinates; namely, $\rmR_3(f) = 5$. By \cite[Theorem 3.18]{CCO17:WaringLoci}, there are no minimal Waring decompositions of $f$ involving $x_0^3$, or equivalently if $\bbX$ is a set of $5$ points apolar to $f$, then $[x_0] \notin \bbX$. Consider the set of points $\bbY$ defined by
\[
 I_\bbY = (y_1y_2, y_0^2y_2+y_0y_2^2+y_2^3, y_0^2y_1-y_0y_1^2) ;
\]
one can check that $I_\bbY \subseteq \Ann_{2}(\nabla f)$ and 
\[
\bbY = \{[x_0],[x_1],[x_0-x_1],[(\omega+1)x_0-2x_2],[(-\omega+1)x_0-2x_2]\} \subseteq \bbP V ,
\]
where $\omega^2 + 3 = 0$. Explicitly, we have 
\[
\begin{array}{rcl}
		\frac{\partial f}{\partial x}= & 2x_0x_1+x_2^2 & = {\scriptsize x_1^2 - (x_0-x_1)^2 + 2 x_0^2 +\frac{\omega+3}{24} \bigl((\omega+1)x_0-2x_2\bigr)^2}+ \\
& & + \frac{3-\omega}{24} \bigl((-\omega+1)x_0-2x_2 \bigr)^2; \\		
& &		\\
		\frac{\partial f}{\partial y}= & x_0^2 & = x_0^2; \\		
& &		\\
		\frac{\partial f}{\partial z}= & 2x_0x_2 & =  x_0^2 +\frac{\omega}{12}\bigl((\omega+1)x_0-2x_2\bigr)^2 - \frac{\omega}{12}\bigl((-\omega+1)x_0-2x_2\big)^2;
	\end{array}
\]

This shows that $\bbY$ defines a simultaneous decomposition of $\nabla f$ containing the point $[x_0]$ which is forbidden for $f$.
\end{example}

\begin{proposition}\label{prop: gradient cubic surfaces}
 Let $f \in S^3 V$ with $\dim V = 4$. Then 
 \[
\cR_{3}(f) = \cR_{2,1} = \gradc{}(f).
 \]
\end{proposition}
\begin{proof}
Recall that $\cR_3(f) \leq 5$ (see e.g. \cite{BuBu15:OnDifference}). If the first catalecticant of $f$ is not full-rank, then there is a choice of coordinates such that $f$ can be written in fewer variables; in this case, the result follows from Proposition \ref{prop: gradient plane cubics}. Therefore suppose that the first catalecticant is full-rank, or equivalently $\HF(A_{\Ann_3(f)};1) = 4$. 
 
Let $\bbX $ be a $0$-dimensional scheme apolar to $\nabla f$, so that by apolarity $I_\bbX \subseteq \Ann_{2}(\nabla f)$. Since $(\Ann_3(f))_{1}  = (\Ann_3(\nabla f))_{1}$, we obtain the lower bound
\[
\gradc{}(f) \geq \HF(A_{\Ann_3(f)};1) = 4, 
\]
and therefore $\deg(\bbX) \geq 4$, providing the result whenever $\cR_3(f) \leq 4$.

If $\deg(\bbX) = 4$, by Remark \ref{rmk: regularity} the ideal $I_\bbX$ is generated by quadrics, and therefore $\bbX$ is apolar to $f$ because $(\Ann_3(f))_{2}  = (\Ann_3(\nabla f))_{2}$ by Proposition \ref{prop: apolar gradient}. This shows that if $\cR_3(f) = 5$, and $\bbX$ is apolar to $\nabla f$, then $\deg(\bbX) \geq 5$. This concludes the proof.
\end{proof}

\subsection{Monomials}
We consider the case of monomials. Recall the result on Waring rank.
\begin{theorem}[\cite{CarCatGer:SolutionMonomials}]\label{thm: solution Waring monomials}
	Let $f = \bfx^\alpha$ with $\alpha_0 = \min_i \{\alpha_i\}$. Then
	$$
		\rk_d(f) = \frac{1}{\alpha_0+1}\prod_{i=0}^n (\alpha_i+1).
	$$
\end{theorem}
Our first goal is to establish that the rank of a monomial coincides with the $k$-th gradient rank, for $k$ at most as large as the minimal exponent appearing in the monomial.

\begin{theorem}\label{thm: gradient monomials}
Let $d \in \bbN$ and $k < d$. Let $f = \bfx^\alpha$ be a monomial with $k\leq \alpha_0 = \min_i\{\alpha_i\}$ and $|\alpha| = d$. Then, for any $\ud \partinto[m]{d}$ with $d_m = d-k$, we have
\[
\rmR_d(f) = \rk_\ud(f) = \gradR{k}(f).
\]
\end{theorem}
\begin{proof}
By \eqref{eq: intro}, it is enough to show $\gradR{k}(f) \geq \rmR_d(f)$. If $a_0 > k$, consider $y_0^k \circ f \in \nabla^k f$. We have $y_0^k \circ f = x_0^{\alpha_0 - k} x_1^{\alpha_1} \cdots x_n^{\alpha_n}$ with $\alpha_0 - k > 0$. Therefore, by Theorem \ref{thm: solution Waring monomials}, $\rmR_{d-k}(y^k \circ f) = \rmR_d(f)$. In particular $\gradR{k}(f) \geq \rmR_{d-k}(y_0^k \circ f)=\rmR_d(f)$ and we conclude.

Assume $a_0 = k$. Let $\XX$ be a minimal set of points apolar to $\nabla^k f$, that is, $I_{\XX}\subseteq \Ann_{d-k}(\nabla^k f)$ and $|\bbX| = \gradR{k}(f)$.  By \eqref{eq: intro}, $|\XX|\leq \rmR_d(f)$. We will show that this inequality cannot be strict. Let $\XX'\subseteq \XX$ be the set of points defined by $I_{\XX'} = I_{\XX}:(y_0)$, i.e., $\XX' = \bbX \setminus \{ y_0 = 0\}$. Therefore,
\[
I_{\bbX'} + (y_0) = I_{\bbX}:(y_0)+(y_0) \subseteq \Ann_{d-k}(\nabla^k f):(y_0) + (y_0).
\]
By Lemma \ref{lemma: basics HF},
\begin{equation}\label{eq: inequality cardinality HF}
|\bbX'| = \sum_{i\geq 0} \HF(A_{I_\bbX'+(y_0)};i) \geq  \sum_{i\geq 0} \HF(A_{\Ann_{d-k}(\nabla ^k f):(y_0) + (y_0)};i).
\end{equation}
Recalling $\alpha_0  = k$, by Proposition \ref{prop: apolar gradient}, we get
\begin{equation*}\label{eq: colon ideal}
\Ann_{d-k}(\nabla^k f):(y_0) = (y_0^{k},y_1^{\alpha_1+1},\ldots,y_n^{\alpha_n+1}) + \left( \bfy^{\beta-\epsilon_0} ~:~ |\beta| = d-k+1, \beta \leq \alpha\right),
\end{equation*}
where $\epsilon_0 = (1,0,\ldots,0)$. Hence,
\begin{align*}
\Ann_{d-k}(\nabla^k f):(y_0) + (y_0) &= (y_0,y_1^{\alpha_1+1},\ldots,y_n^{\alpha_n+1}) + \left( \bfy^{\beta'}: \begin{array}{c}|\beta'| = d-k, \beta' \leq \alpha, \\ \beta'_0 = 0\end{array}\right) =\\ 
 &= (y_0,y_1^{\alpha_1+1},\ldots,y_n^{\alpha_n+1},y_1^{\alpha_1}\cdots y_n^{\alpha_n}).
\end{align*}
From \eqref{eq: inequality cardinality HF} and Theorem \ref{thm: solution Waring monomials}, we derive
\begin{equation}\label{eq: -1 inequality}
	|\bbX'| \geq \frac{1}{\alpha_0+1}\prod_{i=1}^n (\alpha_i+1) - 1 = \rk_d(f) -1.
\end{equation}
Now, if $|\bbX'| > \rk_d(f) -1$ or $|\bbX'| = \rk_d(f)-1$ and $\bbX' \subsetneq \bbX$, then $|\bbX| \geq \rk_d(f)$ and we conclude. Thus, assume $\bbX' = \bbX$ and $|\bbX'| = |\bbX| = \rk_d(f) - 1$. Since
\[
\HF(A_{I_\bbX+(y_0)};i) \geq \HF(A_{\Ann_{d-k}(\nabla^k f):(y_0)+(y_0)};i) \geq 0, \quad \text{ for any } i,
\]
with the constraint
\[
\sum_{i\geq 0} \HF(A_{I_\bbX+(y_0)};i)  = \sum_{i\geq 0} \HF(A_{\Ann_{d-k}(\nabla^kf):(y_0)+(y_0)};i),
\]
we obtain $\HF(A_{I_\bbX+(y_0)};i) = \HF(A_{\Ann_{d-k}(\nabla^k f):(y_0)+(y_0)};i)$. Since
\[
 \bigl( \Ann_{d-k}(\nabla^k f):(y_0)+(y_0) \bigr)_{d-k} = (y_0,y_1^{\alpha_1+1},\ldots,y_n^{\alpha_n+1},y_1^{\alpha_1}\cdots y_n^{\alpha_n}) _{d-k} = S^{d-k} V^*,
\]
we deduce $\HF(A_{I_\bbX+(y_0)};d-k) = 0 $. By Lemma \ref{lemma: basics HF}, we have
\[
\HF(A_{I_\bbX};d-k) = \HF(A_{I_\bbX};d-k-1).
\]
This implies $\reg(\bbX) \leq d-k-1$. Thus, by Remark \ref{rmk: regularity}, we have that the maximal degree of a minimal set of generators of $I_\bbX$ is at most $d-k$. Now, by Proposition \ref{prop: apolar gradient}, $\Ann_{d-k}(\nabla^k f)$ coincides with $\Ann_d(f)$ up to degree $d-k$, so if $I_\bbX$ is generated in degree at most $d-k$, we obtain $I_\bbX \subseteq \Ann_d(f)$. This is a contradiction by Apolarity Lemma. 
\end{proof}

\begin{remark}
The approach adopted in the proof of Theorem \ref{thm: gradient monomials} adapts the approach used in the proof of Theorem \ref{thm: solution Waring monomials} in \cite{CarCatGer:SolutionMonomials} to the case of gradient rank. The same strategy is used in \cite{CCCGW15} to compute the ranks of so-called \textit{$1$-computable forms}; see \cite[Definition 3.5]{CCCGW15}. We observe that our strategy does not necessarily compute the gradient rank of $1$-computable forms in general. For example, in \cite[Proposition 4.4]{CCCGW15} the authors show that $f = x_0^a(x_1^b+\ldots+x_n^b)$ has rank equal to $(a+1)n$. However, in this case the quotient over the ideal $\Ann_{d-1}(\nabla f):(x_1,\ldots,x_n)+(\ell)$ has in general dimension much smaller than $(a+1)n$. This is to stress that inequality \eqref{eq: -1 inequality} is peculiar to the case of monomials and, despite the structure of the proof of Theorem \ref{thm: gradient monomials}, it does not seem to be related to $1$-computability.
\end{remark}

We obtain a similar result about cactus gradient ranks of monomials. Recall the result on the cactus rank of monomials.

\begin{theorem}[\cite{RanSch11:RankForm}, Corollary 2]\label{cactus monomials}
Let $f = \bfx^\alpha$ with $\alpha_n = \max_i\{\alpha_i\}$. Then
\[
	\cactusR_d(f) = \frac{1}{\alpha_n+1}\prod_{i=0}^n (\alpha_i+1).
\]
\end{theorem}

This is obtained by using the following general lower bound, which is proven in \cite[Proposition 1]{RanSch11:RankForm} in a slightly less general setting. 

\begin{lemma}\label{ranestadschreyer}
Let $A_J = \nicefrac{S^\bullet V^*}{J}$ be a graded Artinian algebra and let $I_\bbX \subseteq J$ be an ideal defining a $0$-dimensional scheme $\bbX \subseteq \bbP V$. Let $\delta = \min\{i : J_i \text{ is base point free} \}$. Then
$$
\deg(\bbX) \geq \frac{\dim_{\Bbbk} A_J}{\delta}.
$$
\end{lemma}
\begin{proof}
Let $\widehat{\bbX} \subseteq V$ be the affine cone defined by $I_{\bbX}$; since $\dim \bbX = 0$, we have $\dim \hat{X} = 1$. Let $g \in J_\delta$ be a generic form and let $Z(g) \subseteq V$ be the affine variety defined by the form $g$. Since $J_\delta$ is base point free, $g$ does not vanish on $\bbX$ by Bertini's Theorem \cite[Theorem 8.18]{Hart}. Moreover, by the genericity assumption, $Z(g)$ intersects $\widehat{\bbX}$ properly, namely $\dim (Z(g) \cap \hat{\bbX}) = 0$. Let $\Spec(A_J)$ be the scheme in $V$ defined by $J$, which is a $0$-dimensional scheme supported at $0 \in V$ with $\deg(\Spec(A_J)) = \dim_{\Bbbk}(A_J)$. We have $\Spec(A_J) \subseteq \widehat{\bbX}$ and $\Spec(A_J) \subseteq Z(g)$, therefore $\Spec(A_J) \subseteq Z(g) \cap \widehat{\bbX}$ and since they are $0$-dimensional we obtain $\deg(\Spec(A_J))  \leq \deg(Z(g) \cap \widehat{\bbX})$. By B\'ezout's Theorem, $\dim_{\Bbbk} (A_J) \leq \deg(g) \cdot \deg(\widehat{\bbX}) = \delta \deg(\bbX)$, that concludes the proof.
\end{proof}

A direct consequence of Lemma \ref{ranestadschreyer} and Proposition \ref{prop: apolar gradient} is as follows. For $f \in S^dV$ and every $k < d$, we have
\begin{equation}\label{eq: bound cactus gradient rank}
\gradc{k}(f) \geq \frac{\dim_{\Bbbk} A_{\Ann_{d-k}(\nabla^k f)}}{\delta}, 
\end{equation}
where $\delta := \min \left\{ i : (\Ann_{d-k}(\nabla^k f))_i \text{ is base-point free} \right\}$. In particular, $\delta \leq d-k+1$, because $\left( \Ann_{d-k}(\nabla^k f) \right)_{d-k+1} = S^{d-k+1}V^*$ by Proposition \ref{prop: apolar gradient}. More generally, if $I$ is a graded Artinian ideal, the component $I_\delta$ of degree $\delta$ is base point free if and only if the ideal $(I_\delta)$ that it generates is Artinian.

From this inequality, we derive the following result on cactus gradient ranks of monomials.

\begin{theorem}\label{thm: cactus gradient monomials}
Let $d \in \bbN$ and $n \geq 1$. Let $f = \bfx^\alpha$ with $|\alpha| = d$. Then, for any $\ud \partinto{d}$,
\[
	\cactusR_d(f) = \cactusR_\ud(f) = \gradc{}(f).
\]
\end{theorem}
\begin{proof}
We assume that $\alpha_n = \max_i\{\alpha_i\}$. By Theorem \ref{cactus monomials}, we have that the cactus rank of the monomial is $\cactusR_d(f) = \frac{1}{\alpha_n+1}\prod_{i=0}^n (\alpha_i+1)$ and by Corollary \ref{corol: cactus gradient rank as segre-veronese} we have $\gradc{}(f) \leq \cactusR_d(f)$; we show the opposite inequality. Since $\alpha_n+1 \leq d$, by \eqref{eq: bound cactus gradient rank},
\[
\gradc{}(f) \geq \left\lceil \frac{\prod_{i=0}^n (\alpha_i+1) - 1}{\alpha_n+1}\right\rceil  = \frac{\prod_{i=0}^n (\alpha_i+1)}{\alpha_n+1}.
\]
\end{proof}
We conclude this section with some other remarks about $k$-th gradient ranks of monomial for $k$ sufficiently larger than the minimal exponent.
\begin{lemma}\label{lemma: surjective cat}
Let $f \in S^d V$ and assume that its $k$-th catalecticant matrix {\rm(}see \eqref{def:catalecticant}{\rm)} is surjective. Then $\gradR{k}(f) = \binom{d-k+n}{n}$.
\end{lemma}
\begin{proof}
By assumption, $\langle \nabla^k f \rangle = S^{d-k} V$. Then $\gradR{k}(f) \geq \binom{d-k+n}{n}$. On the other hand, since the Veronese variety $\nu_{d-k}(\bbP V)$ is non-degenerate, we can find a set of points of $\nu_{d-k}(\bbP V)$ which is a basis of the ambient space. Then $\gradR{k}(f) \leq \binom{d-k+n}{n}$.
\end{proof}
\begin{corollary}
Let $f = \bfx^\alpha$ be a monomial with $\alpha_0 = \min_i\{\alpha_i\}$. Let $k$ be an integer such that $k \geq d-a_0$. Then $\gradR{k}(f) = \binom{d-k+n}{n}$.
\end{corollary}
\begin{proof}
We show that $\cat_k(f): S^k V^* \to S^{d-k} V$ is surjective. Every monomial $f' \in S^{d-k} V$ occurs as a $k$-th partial derivative of $f$. Indeed, for any $f' = \bfx^{\beta}$, where $\beta = (\beta_0,\ldots,\beta_n)$ with $|\beta|= d-k$, we have $\beta_j \leq d-k \leq \alpha_0\leq \alpha_j$, for every $j \geq 0$. By Lemma \ref{lemma: surjective cat}, we conclude the proof.
\end{proof}

\subsection{Elementary symmetric polynomials}
In this section, we focus on elementary symmetric polynomials. Let $e_{n+1,d}$ denote the elementary symmetric polynomial of degree $d$ in $n+1$ variables, that is the sum of all square-free monomials of degree $d$, i.e.,
\[
	e_{n+1,d} = \sum_{0 \leq i_1 < \cdots < i_d \leq n} x_{i_1}\cdots x_{i_d} \in S^dV.
\] 
In \cite{Lee16}, Lee determined $\rmR_d(e_{n+1,d})$ for $d$ odd and gave bounds when $d$ is even. 
\begin{theorem}\cite[Theorem 3.4 and Corollary 4.4]{Lee16}\label{thm: waring symmetric poly}
Let $d \in \bbN$ and let $n \geq 1$.

If $d$ is odd, then
\[
	\rk_d(e_{n+1,d}) = \sum_{i=0}^{\frac{d-1}{2}} {\binom{n+1}{i}}.
\]
	If $d$ is even, then
\[
 \sum_{i=0}^{\frac{d}{2}} \binom{n+1}{i} \geq \rmR_d(e_{n+1,d}) \geq \sum_{i=0}^{\frac{d}{2}} \binom{n+1}{i}-\binom{n}{\frac{d}{2}}-1 .
\]
\end{theorem}

We extend these results to the first gradient rank of $e_{n+1,d}$. 
\begin{theorem}\label{thm: gradient elementary symmetric}
Let $d \in \bbN$ and let $n \geq 1$. 

If $d$ is odd, then
\[
	\rmR_d(e_{n+1,d}) = \rmR_{1,d-1}(e_{n+1,d}) = \gradR{}(e_{n+1,d}).
\]
If $d$ is even, then
\[
\rmR_d(e_{n+1,d}) \geq  \gradR{}(e_{n+1,d}) \geq \sum_{i=0}^{\frac{d}{2}} \binom{n+1}{i}-\binom{n}{\frac{d}{2}}-1.
\]
\begin{proof}
By \eqref{eq: intro} and Theorem \ref{thm: waring symmetric poly}, it is enough to prove the lower bounds on $\gradR{}(e_{n+1,d})$. 

By Proposition \ref{prop: apolar gradient}, we have the equality
\begin{equation*}
\Ann_{d-1}(\nabla e_{n+1,d}) = \Ann_d(e_{n+1,d}) + (S^{d} V^*).
\end{equation*}
Let $\phi = \bfy^\beta$ be any square-free monomial: notice that $\phi \circ e_{n+1,d} \neq 0$, and therefore 
\begin{equation}\label{eq: symmetric poly}
\Ann_{d-1}(\nabla e_{n+1,d}) = \Ann_d(e_{n+1,d}) + (\phi).
\end{equation}
Consider a monomial $\phi = \bfy^\beta$ divisible by $y_0$, so $\phi = y_0 \tilde{\phi}$. We are going to show that
\begin{equation}\label{eq: sym poly grad colon}
 \Ann_{d-1}(\nabla e_{n+1,d}): (y_0) = \Ann_d(e_{n+1,d}):(y_0) + (\tilde{\phi}).
\end{equation}
The containment $\Ann_{d-1}(\nabla e_{n+1,d}): (y_0) \supseteq \Ann_d(e_{n+1,d}):(y_0) + (\tilde{\phi})$ is clear from the definitions. For the converse, let $\psi \in \Ann_{d-1}(\nabla e_{n+1,d}): (y_0) $, so that $y_0 \psi  \in \Ann_{d-1}(\nabla e_{n+1,d})$. By \eqref{eq: symmetric poly}, we have $y_0 \psi = \psi _1 + \psi_2 \cdot y_0 \widetilde{\phi}$, for some $\psi_1\in  \Ann_d(e_{n+1,d})$ and $\psi_2 \in S^\bullet V^*$. Hence, $y_0$ divides $\psi_1$, that is, $\psi_1 = y_0 \widetilde{\psi_1}$. We deduce $\tilde{\psi_1}\in \Ann_d(e_{n+1,d}):(y_0)$. Therefore, $\psi = \tilde{\psi_1} + \psi_2 \tilde{\phi} \in \Ann_d(e_{n+1,d}):(y_0) + (\tilde{f})$. This proves \eqref{eq: sym poly grad colon}.

Note that $\Ann_d(e_{n+1,d}):(y_0) + (y_0)= \Ann_{d-1}(e_{n,d-1}),$ where $e_{n,d-1}$ is the elementary symmetric polynomial in the variables $x_1 \vvirg x_n$. Hence, from \eqref{eq: sym poly grad colon}, we get
\begin{align*}
 \Ann_{d-1}(e_{n+1,d}): (y_0) + (y_0) &= \Ann_{d}(e_{n+1,d}) : (y_0) + (\widetilde{\phi}) + (y_0) = \\ &= \Ann_{d-1}(e_{n,d-1}) + (\tilde{\phi}).
\end{align*}
Let $\XX$ be a minimal set of points apolar to $\nabla e_{n+1,d}$, that is, $I_{\XX}\subseteq \Ann_{d-1} (\nabla e_{n+1,d})$ with $|\bbX| = \gradR{}(f) \leq \rk_d(e_{n+1,d})$. Let $\XX' = \XX \cap \{ y_0 \neq 0\}$, so that $I_{\XX'} = I_{\XX}:(y_0)$. Now, we employ the same strategy as in the proof of Theorem \ref{thm: gradient monomials}. Using \eqref{eq: symmetric poly}, we have
\begin{align}
	|\bbX'| & = \sum_{i\geq 0} \HF(A_{I_{\bbX'}+(y_0)};i) \geq  \sum_{i\geq 0} \HF(A_{\Ann_{d-1}(\nabla e_{n+1,d}): (y_0) +(y_0)};i) = \nonumber \\
	& =  \sum_{i\geq 0} \HF(A_{\Ann_{d-1}(e_{n,d-1}) + (\widetilde{\phi})};i).
	\label{eq: symmetric poly 2}
\end{align}
From the proof of \cite[Theorem 3.4]{Lee16}, for $d$ odd, we have 
\[
	\sum_{i\geq 0} \HF(A_{\Ann_{d-1}(e_{n,d-1})};i) = \sum_{i=0}^{\frac{d-1}{2}} \binom{n+1}{i} =  \rmR_d(e_{n+1,d}).
\] 
Now, $\widetilde{\phi}$ is a square-free monomial of degree $d-1$ not divisible by $y_0$, therefore $\tilde{\phi} \notin \Ann_{d-1}(e_{n,d-1})$. By \eqref{eq: symmetric poly 2}, we obtain
$$
|\XX'| \geq \rmR_d(e_{n+1,d})-1.
$$ 
Applying the same argument as in the last part of the proof of Theorem \ref{thm: gradient monomials}, we conclude that $|\XX| \geq \rmR_d(e_{n+1,d})$, which concludes the proof for $d$ odd. 

By the proof of \cite[Corollary 4.4]{Lee16}, for $d$ even, we have 
\[
\sum_{i\geq 0} \HF(A_{\Ann_{d-1}(e_{n,d-1})};i) = \sum_{i=0}^{\frac{d}{2}} \binom{n+1}{i}-\binom{n}{\frac{d}{2}}. 
\]
Again, since $\widetilde{\phi} \notin \Ann_{d-1}(e_{n,d-1})$, we obtain
\[
|\XX| \geq |\XX'|\geq \sum_{i=0}^{\frac{d}{2}} \binom{n}{i}-\binom{n-1}{\frac{d}{2}}-1, 
\]
which concludes the proof for $d$ even.
\end{proof}
\end{theorem}

\subsection*{Acknowledgements} This project was realized during the Research in Pairs program at CIRM Trento in July 2018. We thank CIRM for financial support and for providing a stimulating environment for the development of the project. We also thank E. Ballico and A. Bernardi for helpful discussions and suggestions during our stay in Trento. Finally, we thank the three anonymous referees for their useful comments.

F.G. acknowledges financial support from the European Research Council (ERC Grant Agreement no. 337603) and VILLUM FONDEN via the QMATH Centre of Excellence (Grant no. 10059). A.O. acknowledges financial support from the Spanish Ministry of Economy and Competitiveness, through the Mar\'ia de Maeztu Programme for Units of Excellence in R$\&$D (MDM-2014-0445). E.V. acknowledges financial support by the grant 346300 for IMPAN from the Simons Foundation and the matching 2015-2019 Polish MNiSW~fund. 

\bibliographystyle{amsalpha}
\begin{small}
	\bibliography{bibGOV}
\end{small}

\end{document}